\documentclass{amsart}
\usepackage{amssymb}
\usepackage[all]{xy}
\usepackage{graphicx}
\usepackage[T1]{fontenc}
\usepackage[bitstream-charter]{mathdesign}
\usepackage{hyperref}
\theoremstyle{plain}
\newtheorem{theorem}{Theorem}[section]
\newtheorem{prop}[theorem]{Proposition}
\newtheorem{lemma}[theorem]{Lemma}
\newtheorem{conj}[theorem]{Conjecture}
\newtheorem{question}[theorem]{Question}
\newtheorem{cor}[theorem]{Corollary}
\newtheorem*{appthm}{Theorem 1.2}
\theoremstyle{remark}
\newtheorem{remark}[theorem]{Remark}
\newtheorem{ex}[theorem]{Example}
\theoremstyle{definition}
\newtheorem{definition}[theorem]{Definition}

\def\co{\colon\thinspace}

\def\ep{\epsilon}
\def\R{\mathbb{R}}
\def\Z{\mathbb{Z}}
\title{Many closed symplectic manifolds have infinite Hofer--Zehnder capacity}
\author{Michael Usher}\address{Department of Mathematics, University of Georgia\\ Athens, GA 30602, USA}\email{usher@math.uga.edu}

\begin{document}

\begin{abstract} We exhibit many examples of closed symplectic manifolds on which there is an autonomous Hamiltonian whose associated flow has no nonconstant periodic orbits (the only previous explicit example in the literature was the torus $T^{2n}$ ($n\geq 2$) with an irrational symplectic structure).  The underlying smooth manifolds of our examples include, for instance: the $K3$ surface and also infinitely many smooth manifolds homeomorphic but not diffeomorphic to it; infinitely many minimal four-manifolds having any given finitely-presented group as their fundamental group; and simply connected minimal four-manifolds realizing all but finitely many points in the first quadrant of the geography plane below the line corresponding to signature $3$.    The examples are constructed by   performing symplectic sums along suitable tori and then perturbing the symplectic form in such a way that hypersurfaces near the ``neck'' in the symplectic sum have no closed characteristics.  We conjecture that any closed symplectic four-manifold with $b^+>1$ admits symplectic forms with a similar property.
\end{abstract}

\maketitle

\section{Introduction}
A smooth, compactly supported function $H\co X\to\mathbb{R}$ on a symplectic manifold $(X,\omega)$ induces a ``Hamiltonian vector field'' $X_H$ by the prescription that $\omega(X_H,\cdot)=-dH$.  If $c$ is a regular value of $H$, then $X_H$ is tangent to the level set $Y=H^{-1}(c)$; indeed the restriction of $X_H$ to $Y$ generates the characteristic line bundle $TY^{\omega}=\{v\in TY|(\forall w\in TY)(\omega(v,w)=0)\}$.  

It is natural to ask whether the flow of $X_H$ has any periodic orbits on the level set $Y$, or equivalently, whether the ``characteristic foliation'' generated by the characteristic line bundle $TY^{\omega}$ has any closed leaves.  In the case that $Y$ has contact type, the existence of such a closed leaf is postulated by the Weinstein conjecture (now a theorem in many cases, for instance whenever $\dim Y=3$ or whenever $X$ is a subcritical Stein manifold).  If no geometric assumption is made on $Y$, however, a closed leaf need not exist: examples of this in $\mathbb{R}^{2n}$ were supplied by Ginzburg \cite{Gi95},\cite{Gi97} for $2n\geq 6$ and (with $H$ only $C^2$) by Ginzburg--G\"urel \cite{GiGu} for $2n=4$; by Darboux's theorem one can consequently find a Hamiltonian $H$ on an arbitrary symplectic manifold of dimension at least six (or four if one allows $H$ to be only $C^2$) for which there is a regular level surface on which $X_H$ has no periodic orbits.

One can instead ask a weaker question: Given a (compactly supported) Hamiltonian $H\co X\to\mathbb{R}$, must $X_H$ have some nonconstant periodic orbit? (Obviously $H$ has constant periodic orbits at its critical points.)  If the answer to this question is affirmative for all Hamiltonians, it is not difficult to show (by postcomposing  Hamiltonians $H$ with  appropriate functions $f\co \R\to\R$) that in fact for any Hamiltonian $H$ the set of regular values $c$ such that the characteristic foliation of $H^{-1}(c)$ has a closed leaf is dense in the range of $H$.  In this case we say that the ``dense existence property'' holds in $(X,\omega)$.

It was shown in \cite{HZ1} that the dense existence property holds in $\mathbb{R}^{2n}$ with its standard symplectic structure.  There are also many closed manifolds in which the dense existence property is known to hold, as we will shortly recall.

In fact, to the best of my knowledge, all closed manifolds $(X,\omega)$ for which the dense existence property holds in fact have finite \emph{Hofer--Zehnder capacity} \[ c_{HZ}(X,\omega)=\sup\left\{\max H-\min H\left|\begin{array}{cc} H\co X\to\mathbb{R}\mbox{ is smooth and all periodic }\\ \mbox{orbits of the Hamiltonian vector} \\ \mbox{field }X_H  \mbox{ of period }\leq 1\mbox{ are constant}\end{array}\right.\right\}.\]  Consequently by \cite[Theorem IV.4]{HZ94}  these manifolds obey a stronger ``almost existence property'': the set of regular values $c$ of $H$ for which $X_H$ has a periodic orbit on $H^{-1}(c)$ has full measure in the range of $H$.

For all known (at least to me) examples of closed symplectic manifolds on which the dense existence property holds, this fact can be proven as follows.
For a natural number $g$, let us say that a symplectic manifold $(X,\omega)$ is \emph{$GW_g$-connected} if there exists a nonvanishing Gromov--Witten invariant of the form 
\[GW_{g,k+2,A}^{X}([pt],[pt],\beta_1,\ldots,\beta_k;B), \] where $[pt]$ denotes the homology class of a point in $X$; $\beta_1,\ldots,\beta_k$ are some other classes in $H_*(X;\Z)$; and $B\in H_*(\bar{M}_{g,k+2},\mathbb{Q})$.  In other words, this invariant should enumerate genus-$g$ pseudoholomorphic representatives of the class $A$ which pass through two generic points and possibly satisfy some other constraints in $X$ represented by the $\beta_i$ and a constraint on the  underlying stable curve represented by $B$.

We have \begin{theorem}\cite[Corollary 1.19]{Lu}\label{luthm} Any closed symplectic manifold $(X,\omega)$ which is $GW_g$-connected has finite Hofer--Zehnder capacity.  Indeed, if $GW_{g,k+2,A}^{X}([pt],[pt],\beta_1,\ldots,\beta_k;B)\neq 0$, then  \[ c_{HZ}(X,\omega)\leq \langle[\omega],A\rangle.\]
\end{theorem}

This criterion applies to a variety of manifolds: for instance it applies to all closed symplectic two-manifolds, and also all closed toric symplectic manifolds (in the latter case one can even take $g=0$).  
In dimension four, an important dichotomy among smooth oriented manifolds $X$ is determined by the characteristic number $b^+(X)$, defined to be the maximal dimension of a subspace of $H^2(X;\R)$ on which the quadratic form $\alpha\mapsto \langle \alpha\cup\alpha,[X]\rangle$ is positive definite.  Of course if $(X,\omega)$ is a symplectic four-manifold the fact that $\int_X\omega\wedge\omega>0$ shows that $b^+(X)\geq 1$. The Seiberg--Witten invariants of $X$ (and hence, by Taubes' celebrated results \cite{Tbook}, also certain Gromov--Witten invariants of $(X,\omega)$) behave rather differently according to whether $b^+(X)=1$ and $b^+(X)>1$.  In particular, we have:

\begin{theorem}[Li-Liu]\label{llthm} Any symplectic four-manifold with $b^+=1$ is $GW_g$-connected for some $g$.  Consequently all such symplectic four-manifolds have finite Hofer--Zehnder capacity.
\end{theorem}

The proof of this theorem merely requires one to straightforwardly assemble various ingredients that are scattered through the literature (mostly in work of Li and Liu, hence the attribution); for the reader's convenience we provide an explicit proof in the appendix.

The present paper, on the other hand, is about closed symplectic manifolds on which the dense existence property does \emph{not} hold; thus there exist Hamiltonians on these manifolds whose flows have no nonconstant periodic orbits.  To avoid excessive use of negatives we make the following definitions:

\begin{definition} \label{aperiodic} Let $X$ be a smooth manifold and let $\omega$ be a symplectic form on $X$.  We say that $\omega$ is \emph{aperiodic} if there exists a compactly supported smooth function $H\co X\to\mathbb{R}$ which is not everywhere locally constant such that the Hamiltonian flow of $H$ has no nonconstant periodic orbits.
\end{definition}
\begin{definition} \label{nearby} Let $(X,\omega)$ be a symplectic manifold and let $Y$ be a coorientable (equivalently, orientable) closed hypersurface in $X$.  We say that $Y$ \emph{violates the nearby existence property} with respect to $(X,\omega)$ if there is a tubular neighborhood $\psi\co (-\ep,\ep)\times Y\hookrightarrow X$ such that for every $s\in (-\ep,\ep)$ the characteristic foliation of the hypersurface $\psi(\{s\}\times Y)$ has no closed leaves.  Otherwise we say that $Y$ satisfies the nearby existence property.
\end{definition}

One has:
\begin{prop} A symplectic form $\omega$ on a closed manifold $X$ is aperiodic if and only if there is a closed coorientable hypersurface $Y\subset X$ which violates the nearby existence property with respect to $(X,\omega)$.
\end{prop}

\begin{proof} If $Y$ violates the nearby existence property, let $\psi\co (-\ep,\ep)\times Y\hookrightarrow X$ be a tubular neighborhood as in Definition \ref{nearby}.  Let $f\co (-\ep,\ep)\to\mathbb{R}$ be any not-identically-zero compactly supported smooth function, and define $H\co X\to\mathbb{R}$ by $H(x)=0$ if $x\notin Im(\psi)$ and $H(\psi(s,y))=f(s)$.  Then $H$ is a smooth function such that $\iota_{X_H}\omega$ vanishes on the tangent space to each $\psi(\{s\}\times Y)$, while $X_H$ is zero outside the image of $\psi$.  So all orbits of $X_H$ are either constant or nonconstant and tangent to a leaf of the characteristic foliation of some $\psi(\{s\}\times Y)$; the latter case occurs precisely when $s$ is not a critical point of $f$, and in this case $X_H$ is nowhere-vanishing along the orbit.  So since the characteristic foliations of the $\psi(\{s\}\times Y)$ are assumed to have no closed leaves, $X_H$ has no nonconstant periodic orbits.

Conversely, suppose that $X_H$ has no nonconstant periodic orbits and let $Y=H^{-1}(z)$ for some regular value $z$ of $H$.  In particular $dH\co TX|_Y/TY\to\mathbb{R}$ provides a coorientation for $Y$.  Since $H$ has compact support, its critical values form a compact subset of $\mathbb{R}$, so for some $\ep>0$ there will be no critical values in the interval $[z-\ep,z+\ep]$.  Using the gradient flow of $H$ with respect to a suitable metric (and the compactness of the support of $H$) we can construct a tubular neighborhood $\psi\co (-\ep,\ep)\times Y\to H^{-1}(z-\ep,z+\ep)$ so that $\psi(\{s\}\times Y)=H^{-1}(z+s)$.  Since $dH$ vanishes nowhere on the image of $\psi$, the Hamiltonian vector field will generate the characteristic foliation of each  $\psi(\{s\}\times Y)$ (in particular it will be nonzero everywhere on $\psi(\{s\}\times Y)$).  The leaf of the characteristic foliation through a point of $\psi(\{s\}\times Y)$ thus coincides with the (nonconstant) orbit of the point under $X_H$; since none of these orbits are periodic none of the leaves are closed.
\end{proof}

If $\omega$ is aperiodic then clearly $(X,\omega)$ necessarily has infinite Hofer--Zehnder capacity; conversely I do not know of any closed manifolds $(X,\omega)$ that are known to have infinite Hofer--Zehnder capacity without $\omega$ being aperiodic, though in principle this seems possible (if one drops the assumption that the manifold is closed then $\mathbb{R}^{2n}$ provides an example.)

There is essentially one example of an aperiodic closed symplectic manifold in the previous literature: what is known as ``Zehnder's torus'' in honor of \cite{Ze}, namely the torus $T^{2n}=(\mathbb{R}/\mathbb{Z})^{2n}$ ($2n\geq 4$) equipped with a constant-coefficient symplectic form $\omega_Z=\sum A_{ij}dx^i\wedge dx^j$ some of whose coefficients are linearly independent over $\mathbb{Q}$.  For suitable choices of the $A_{ij}$ each of the hypersurfaces $\{x_{2n}=c\}$ will have characteristic foliation generated by a vector field having some rationally independent coordinates; of course such a vector field has no periodic orbits.

In this paper we show that manifolds admitting aperiodic symplectic forms form an extremely diverse class from a topological standpoint:
\begin{theorem}\label{main}  There exist aperiodic symplectic forms on each of the following classes of closed manifolds:
\begin{itemize} \item[(i)] The elliptic complex surface $E(n)_{p,q}$ for $n\geq 2$ and $p,q\in\mathbb{Z}_+$ (where $E(2)_{1,1}$ is the $K3$ surface), and also each of the Fintushel--Stern \cite{FS} knot surgery manifolds $E(n)_K$ where $K$ is any fibered knot and $n\geq 2$.
\item[(ii)] For any nonempty finite presentation of a group $G$, the four-manifold $X_G$ produced in \cite[Theorem 4.1]{Gom} having $\pi_1(X_G)=G$, and also the minimal four-manifolds of \cite[Theorem 6.2]{Gom} having fundamental group $G$ and prescribed Euler characteristic and signature.
\item[(iii)] For all but finitely many pairs of integers $(\chi,c)$ obeying $0\leq c\leq 8\chi+2$, simply-connected minimal four-manifolds constructed in \cite{ABBKP} whose Euler characteristic $e$ and signature $\sigma$ obey $\frac{e+\sigma}{4}=\chi$ and $3\sigma+2e=c$.
\item[(iv)] For any linear symplectomorphism $\psi\in SL(2n-2,\mathbb{Z})\cap Sp(2n-2,\mathbb{R})$ of the $(2n-2)$-torus such that $1$ is an eigenvalue of $\psi$, the manifold $X_{\psi}=S^1\times T_{\psi}$ where $T_{\psi}$ is the mapping torus of $\psi$, \emph{i.e.}, \[ T_{\psi}=\frac{\mathbb{R}\times T^{2n-2}}{(t+1,v)\sim (t,\psi(v))}.\] 
\end{itemize}
\end{theorem}
\begin{proof}
See, respectively, Section \ref{elliptic}, Section \ref{arbpi}, Section \ref{geog}, and Example \ref{kt}.
\end{proof}

The manifolds in Theorem \ref{main}(i)-(iii) are standard examples that illustrate the diversity of closed symplectic four-manifolds; thus Theorem \ref{main} shows that similar diversity exists in the class of four-manifolds admitting aperiodic symplectic forms.  For instance, the manifolds $E(n)_K$ for any fixed $n$ but varying $K$ comprise infinitely many mutually nondiffeomorphic smooth manifolds each of which is homeomorphic to $E(n)=E(n)_{1,1}$.  The $E(n)_{p,q}$ with $gcd(p,q)=1$ are also mutually nondiffeomorphic but homeomorphic, and are also nondiffeomorphic to each of the $E(n)_K$.  As $n,p,q$ vary with $gcd(p,q)=1$ and $n\geq 2$, the $E(n)_{p,q}$ comprise all of the simply connected complex minimal\footnote{Recall that a symplectic manifold is called minimal if it cannot be obtained by blowing up another symplectic manifold; it's easy to see that the property of admitting an aperiodic symplectic form is preserved when a manifold is blown up  (perform the symplectic blow up on a ball disjoint from a hypersurface violating the nearby existence property).} elliptic surfaces with $b^+>1$ (see \cite[Chapter 3]{GS} for an introduction to elliptic surfaces).  Allowing the homotopy type to vary, Theorem \ref{main}(ii) and (iii) show that we can also achieve considerable diversity in either or both of $\pi_1$ or $H^2$ while staying within the class of manifolds admitting aperiodic symplectic forms.  We have phrased Theorem \ref{main} to emphasize the point that the manifolds in question already existed in the literature---it was not necessary to modify their topological construction in any way in order to ensure that they admit aperiodic symplectic forms.

It was, however, necessary to modify the standard symplectic forms on these manifolds: each of the manifolds is, topologically speaking, obtained by the symplectic sum operation of \cite{Gom}, \cite{MW}; this operation induces a natural symplectic form on the manifold, but our aperiodic symplectic forms are small perturbations of this standard form.  In particular our forms are deformation equivalent (\emph{i.e.}, homotopic through symplectic forms) to the standard ones.  Related to this, we ask: 

\begin{question} Is the property of being aperiodic invariant under deformations of the symplectic form?
\end{question}

Current methods seem ill-equipped to answer this question either affirmatively or negatively. On the one hand, it is an inherent limitation of our method that the aperiodic symplectic forms that it produces can never have de Rham cohomology classes that lie in the subgroup $H^2(X;\mathbb{Q})\subset H^2(X;\R)$, whereas any symplectic form is deformation equivalent to symplectic forms with cohomology class in $H^2(X;\mathbb{Q})$.  On the other hand, the only way that I know how to prove that a symplectic form on a closed manifold is \emph{not} aperiodic is by applying Theorem \ref{luthm}; since Gromov--Witten invariants are unchanged under deformation of the symplectic form this technique can never yield a non-aperiodic symplectic form that is deformation equivalent to an aperiodic one.



Consistently with Theorem \ref{llthm}, all of the four-manifolds in Theorem \ref{main} have $b^+>1$.  Conversely, all known symplectic four-manifolds with finite Hofer--Zehnder capacity have $b^+=1$. 
Based on the relationship between Gromov--Witten invariants and Hofer--Zehnder capacity, this is not a coincidence,  as symplectic four-manifolds with $b^+>1$ tend not to have nonvanishing Gromov--Witten invariants counting curves satisfying nontrivial incidence conditions.   In fact, \cite[Corollary 3.4]{LP} shows that a K\"ahler surface with $b^+>1$ can \emph{never} have such invariants; consequently, if such a manifold were to have finite Hofer--Zehnder capacity this would need to be established by a method fundamentally different than any currently known.

On the other hand,  this paper produces a multitude of symplectic four-manifolds having $b^+>1$ which admit aperiodic symplectic forms.  The ease with which such examples can be constructed leads me to the following conjecture, which can be thought of as a stronger version of the converse to Theorem \ref{llthm}:

\begin{conj}\label{mainconj} Let $X$ be a closed oriented four-manifold with $b^+(X)>1$ which admits symplectic structures.  Then there exists an aperiodic symplectic form on $X$.
\end{conj}

Note that Theorem \ref{main}(i) shows that the conjecture holds within the class of simply connected elliptic surfaces.  
Even if Conjecture \ref{mainconj} is true, there might well exist symplectic forms on manifolds with $b^+>1$ which are not aperiodic; for instance even on Zehnder's original example of $T^4$ it is a major open question whether the standard symplectic form $dx^1\wedge dx^2+dx^3\wedge dx^4$ has finite Hofer--Zehnder capacity.  It is also conceivable that any \emph{rational} symplectic form (\emph{i.e.}, one whose cohomology class belongs to $H^2(X;\mathbb{Q})$) has finite Hofer--Zehnder capacity; as mentioned earlier the method used in this paper, while producing many forms with infinite Hofer--Zehnder capacity, never produces ones which are rational.

\subsection{Summary of the construction}
The main observation underlying this paper is that the Hamiltonian dynamics appearing in Zehnder's torus can be found in a very wide variety of symplectic manifolds.  Using a standard coisotropic neighborhood theorem, one can see that a symplectic $2n$-manifold $(X,\omega)$ will be aperiodic provided that there is a hypersurface $T\subset X$ diffeomorphic to the $(2n-1)$-torus such that $\omega|_T$ coincides with the restriction of Zehnder's symplectic form $\omega_Z$ to a hypersurface $\{x_{2n}=c\}$.  While it may seem difficult to find examples of such $T\subset X$, we construct many by the following somewhat indirect procedure.  Begin instead with a symplectic $2n$-manifold $(X,\tilde{\omega})$ and a $(2n-1)$-torus $T\subset X$ so that $\tilde{\omega}|_T$ is a $2$-form with constant, possibly rational, coefficients; as will be explained shortly, there are  rich sources of such examples.  
Then, as a form on $T$, $\tilde{\omega}|_T$ has small perturbations $\tilde{\omega}|_T+\ep\eta_T$ which have irrational coefficients and have kernels having no periodic orbits.  Such small perturbations can be extended to forms $\tilde{\omega}+\ep\eta$ on all of $X$ provided that their de Rham cohomology classes lie in the image of the inclusion-induced map $H^2(X;\R)\to H^2(T;\R)$; in particular some of these perturbations will  extend if the rank of $H^2(X;\R)\to H^2(T;\R)$ is two or more.  In this case, at least if $X$ is compact, the form $\omega=\tilde{\omega}+\ep\eta$ will be symplectic for small $\ep$, and will be aperiodic because it restricts to an appropriate constant-coefficient irrational form on $T$.  

We mentioned above that there are many examples of symplectic $(X,\tilde{\omega})$ containing torus hypersurfaces $T$ with $\tilde{\omega}|_T$ a constant-coefficient form; here is how we produce these.  More specifically, we will obtain forms $\tilde{\omega}$ on $X$ such that, where $V=T^{2n-2}$ and $p\co T\to V$ is the projection which removes the last coordinate, we have $\tilde{\omega}|_T=p^*\omega_V$ for some constant-coefficient form $\omega_V$ on $V$.  (Below we call such a hypersurface $T\subset X$ ``$\omega_V$-embedded.''  In particular the characteristic foliation of $\tilde{\omega}$ on $T$ is given by the fibers of the fibration and so all of its leaves are closed, whether or not $\omega_V$ has irrational coefficients.)  This situation is produced whenever one performs a \emph{symplectic sum} (\cite{Gom},\cite{MW}) on a (possibly disconnected) symplectic manifold $(M,\omega')$ equipped with embeddings of symplectic $(2n-2)$-tori $i_{\pm}\co V\to M$ with $i_{\pm}^{*}\omega'=\omega_V$ where the $i_{\pm}(V)$ have trivial normal bundle.  In fact, there is a converse statement: whenever one has an $\omega_V$-embedded hypersurface $T\subset X$ one can form (following \cite{Le}, though with somewhat different conventions) a symplectic manifold $(M,\omega')$ called the symplectic cut of $X$ along $T$, and then $X$ and $T$ can be recovered (up to $T$-preserving symplectomorphism) by suitably applying the symplectic sum to $M$.

Thus a basic strategy for producing aperiodic symplectic manifolds $(X,\omega)$ can be described as follows: start with a symplectic manifold $(M,\omega')$ containing disjoint codimension-two symplectic embeddings of tori $i_{\pm}\co V\to M$ so that $i_{\pm}^{*}\omega'=\omega_V$ for a constant-coefficient symplectic form $\omega_V$.  Then apply the symplectic sum, yielding a symplectic manifold $(X,\tilde{\omega})$ with an $\omega_V$-embedded torus hypersurface $T$.  While the characteristic foliation on $T$ given by $\tilde{\omega}$ will consist entirely of closed leaves, if the map $H^2(X;\R)\to H^2(T;\R)$ has rank at least two then $\tilde{\omega}$ will admit symplectic perturbations $\omega$ which induce on $T$ the dynamics of Zehnder's torus; in particular $\omega$ will be aperiodic.  

There thus remains the matter of determining when the inclusion-induced map $i^*\co H^2(X;\R)\to H^2(T;\R)$ has rank at least two, which in the present case is equivalent to $p_{!}\circ i^*\co H^2(X;\R)\to H^1(V;\R)$ having rank at least one where $p_!\co  H^2(T;\R)\to H^1(V;\R)$ is the Gysin map.     Under certain modest hypotheses, Lemma \ref{maincalc} calculates the image of $p_!\circ i^*$ (for a general symplectic sum, not necessarily along tori with trivial normal bundle) entirely in terms of data associated to the symplectic embeddings $i_{\pm}\co V\to M$.  So when this calculation yields a nontrivial result, we obtain aperiodic symplectic forms $\omega$ on the manifold $X$ given by the symplectic sum construction.

While all of the above is valid in arbitrary dimensions, in practice it is easiest to find interesting examples in dimension four, which accounts for the four-dimensional focus of Theorem \ref{main}.  For one thing, it is easier to find codimension-two symplectic tori along which to perform the symplectic sum when the ambient dimension is four: the tori need to be symplectomorphic, but since the genus and the area provide complete invariants of symplectic two-manifolds codimension-two tori in a symplectic four-manifold will be symplectomorphic as soon as they have the same area.  In the case where the manifold $M$ is disconnected and the tori lie on different connected components, even this area constraint is not really significant, since the tori will become symplectomorphic after the symplectic form on one of the connected components of $M$ is rescaled.  Also, in dimension four the criterion provided by Lemma \ref{maincalc} is more easily satisfied.  In particular, when the tori $i_{\pm}(V)$ have self-intersection zero and lie on different connected components of $M$, Lemmas \ref{maincalc} and \ref{irrat} show that we can obtain aperiodic symplectic forms on the sum $X$ provided that the kernels of the inclusion induced maps $(i_{\pm})_*\co H_1(V;\R)\to H_1(M;\R)$ have nontrivial intersection.  Not coincidentally in light of Theorem \ref{llthm},  one can check that this latter condition also ensures that the sum $X$ will have $b^+>1$ (in particular, while Theorem \ref{main} (iii) includes many manifolds from \cite{ABBKP}, there are also some manifolds in \cite{ABBKP} that are formed by symplectic sums along tori and have $b^+=1$ to which our method does not apply because this condition on the kernels of $(i_{\pm})_*\co H_1(V;\R)\to H_1(M;\R)$ does not hold).

In dimension four,  one can consider more general symplectic sums along tori in $M$, where $i_{\pm}(V)$ have opposite, nonzero self-intersection $\pm k$.  This produces a hypersurface $Y\subset X$ where $X$ is the symplectic sum, with $Y$ diffeomorphic to the principal $S^1$-bundle over $T^2$ with Euler number $k$.  Lemma \ref{irrat} produces aperiodic symplectic forms on $X$ (with the characterstic foliation on $Y$ having no closed leaves) provided that the map $p_!\circ i^*\co H^2(X;\R)\to H^1(V;\R)$ is surjective.  If the tori lie on different connected components of $M$, Lemma \ref{maincalc} shows that this condition holds if the maps $(i_{\pm})_*\co H_1(V;\R)\to H_1(M;\R)$ both vanish.

Of course, whenever one finds a hypersurface $Y\subset X$ which violates the nearby existence property, it is also true that, for any closed symplectic manifold $P$, the hypersurface $Y\times P\subset X\times P$ violates the nearby existence property.  Thus if $(X,\omega)$ is aperiodic then so is $(X\times P,\omega\oplus \sigma)$ for any closed symplectic manifold $(P,\sigma)$.  Thus our diverse examples of four-dimensional aperiodic symplectic manifolds give rise to many examples of aperiodic symplectic manifolds in arbitrary even dimension larger than four.  Applying this to the manifolds $E(n)_K$ of Theorem \ref{main}(i), as discussed in Remark \ref{prodrem} this produces infinite families of mutually diffeomorphic but non-deformation-equivalent symplectic six-manifolds all admitting aperiodic symplectic forms. 

Finally, we mention that our main construction is not the only one capable of producing aperiodic symplectic forms.  Here is another one, a special case of which appears at the end of the paper in Section \ref{nontrivy}.  Suppose that $(U,\sigma)$ is a closed symplectic manifold and $\phi\co U\to U$ is a symplectomorphism having no periodic points.  Then where $T_{\phi}$ is the mapping torus of $\phi$, there is a natural symplectic form on $S^1\times T_{\phi}$ with respect to which each of the hypersurfaces $\{s\}\times T_{\phi}$ violates the nearby existence property.  Thus symplectomorphisms without periodic points give rise to aperiodic symplectic forms two dimensions higher.  In particular, the Kodaira--Thurston manifold can be obtained by this construction, and we obtain symplectic forms on the Kodaira--Thurston manifold admitting hypersurfaces diffeomorphic to the Heisenberg manifold (\emph{i.e.}, the principal $S^1$-bundle over $T^2$ with Euler number $\pm 1$, depending on orientation) which violate the nearby existence property.  While our main construction does in certain cases give hypersurfaces diffeomorphic to the Heisenberg manifold which violate the nearby existence property, we prove in Section \ref{nontrivy} that this particular case in the Kodaira--Thurston manifold cannot be obtained in this fashion.  

There are, incidentally, aperiodic symplectic forms on the Kodaira--Thurston manifold that are produced by our main construction; indeed this is a special case of Theorem \ref{main}(iv).  However in these cases the hypersurface violating the nearby existence property is a $3$-torus, not the Heisenberg manifold.

\subsection{Outline of the paper}

Section \ref{pert} contains the main dynamical input to the construction, Lemma \ref{irrat}, stating that if a closed symplectic manifold $(X,\omega)$ contains a hypersurface $Y$ diffeomorphic to a principal bundle over a torus with $\omega|_Y$ suitably compatible with the bundle projection, then under a topological hypothesis on the map $H^2(X;\R)\to H^2(Y;\R)$ the symplectic form $\omega$ will have perturbations with respect to which $Y$ violates the nearby existence property.

Section \ref{surg} contains a review of two symplectic surgery operations, the symplectic cut and the symplectic sum.  When $Y\subset X$ is a hypersurface as in the previous paragraph, one can form a new symplectic manifold $(M,\omega')$ containing distinguished codimension-two symplectic submanifolds by cutting $X$ along $Y$ and then collapsing the circle fibers of the resulting boundary components.  Conversely, starting with two suitably compatible codimension-two symplectic submanifolds of a (perhaps disconnected) symplectic manifold, one can form the symplectic sum; topologically this operation amounts to removing neighborhoods of the submanifolds and then gluing the resulting boundary components.  We describe these constructions in detail, in order to draw the conclusion that they are inverses to each other up to symplectomorphisms \emph{which preserve the additional data}, including the hypersurface $Y$; this is a stronger inversion statement than seems to appear in the literature, though it should not surprise experts.  

Section \ref{calc} connects the previous two; while Lemma \ref{irrat} gives a topological condition on $Y\subset X$ which suffices for the existence of aperiodic symplectic forms on $X$, the main result of Section \ref{calc}, Lemma \ref{maincalc}, allows one to check this topological condition on $X$ in terms of topological data in the manifold $M$ formed by applying the symplectic cut to $X$ along $Y$.  Thus, conversely, if we form $X$ by applying the symplectic sum to a manifold $M$ on which we have arranged the topological condition to hold, then $X$ will admit aperiodic symplectic forms.  This is the principle behind most of our examples.

With this preparation, in Section \ref{ex} we are able to provide the examples yielding Theorem \ref{main}, in each case recalling the relevant manifolds from the literature and verifying that they can be formed by a symplectic sum of the sort that satisfies the topological conditions established in Lemmas \ref{irrat} and \ref{maincalc}.

Finally, the Appendix contains the promised proof of Theorem \ref{llthm}.

\subsection*{Acknowledgements} I thank Viktor Ginzburg for his interest in and comments on this work.

\section{Perturbing to an aperiodic form}\label{pert}

In this paper we identify the $S^1$ with the unit circle in the complex plane.  When we instead wish the circle to have length one we will denote it by $\mathbb{R}/\mathbb{Z}$.

\begin{definition} Let $(V,\omega_V)$ be a symplectic $(2n-2)$-manifold, $p\co Y\to V$ be a principal $S^1$-bundle over $V$, and let $i_Y\co Y\hookrightarrow X$ be an embedding of $Y$ into a symplectic $2n$-manifold $(X,\omega)$.  We say that $Y\subset X$ is \emph{$\omega_V$-embedded} in $X$ if we have $i_{Y}^{*}\omega=p^*\omega_V$ for some volume form $\eta$ on $T$.
\end{definition}

Recall that for a principal $S^1$-bundle $p\co Y\to V$ we have a Gysin map $p_{!}\co H^{k}(Y;\mathbb{R})\to H^{k-1}(V;\mathbb{R})$ (given in terms of de Rham cohomology by integration down the fiber, or equivalently by transferring the map induced on homology by $p$ to cohomology via Poincar\'e duality on both $Y$ and $V$).  This map fits into an exact sequence (the ``Gysin sequence'')
\begin{equation}\label{gysin}  \xymatrix{  H^k(V;\mathbb{R})\ar[r]^{p^*} & H^k(Y;\mathbb{R})\ar[r]^{p_{!}} & H^{k-1}(V;\mathbb{R})\ar[r]^{e\cup\cdot} & H^{k+1}(V;\mathbb{R})}\end{equation} where $e\in H^2(V;\mathbb{R})$ is the Euler class of the bundle $p\co Y\to V$.

The following lemma shows that if $Y\subset X$ is $\omega_V$-embedded where $V$ is a torus, then in certain cases $X$ admits symplectic forms with respect to which $Y$ violates the nearby existence property provided that the image of the composition $p_!\circ i_{Y}^{*}\co H^2(X;\R)\to H^1(V;\R)$ is large enough.  Together with Lemma \ref{maincalc}, which calculates $Im(p_!\circ i_{Y}^{*})$ in some cases, this will serve as the basis for our examples in Theorem \ref{main}.

\begin{lemma}\label{irrat}  Suppose that $V=\{(x_1,\ldots,x_{2n-2})|x_i\in\mathbb{R}/\mathbb{Z}\}$ is a $(2n-2)$-dimensional torus with a linear symplectic form $\sum_{i,j=1}^{2n-2}B_{ij}dx^i\wedge dx^j$ (with the $B_{ij}$ constant) and that $p\co Y\to V$ is $\omega_V$-embedded in the closed symplectic $2n$-manifold $(X,\omega)$ via the map $i\co Y\to X$.  Suppose moreover that either \begin{itemize}\item[(i)] The bundle $p\co Y\to V$ is trivial and the composition $p_{!}\circ i^*\co H^2(X;\mathbb{R})\to H^2(Y;\mathbb{R})\to H^1(V;\mathbb{R})$ is nonzero, or \item[(ii)] $\dim V=2$ and the composition $p_{!}\circ i^*\co H^2(X;\mathbb{R})\to H^2(Y;\mathbb{R})\to H^1(V;\mathbb{R})$ is surjective.
\end{itemize}  Then there exists a smooth family of symplectic forms $\{\omega_u\}_{u\in (-\delta,\delta)}$ such that $Y$ violates the nearby existence property with respect to $\omega_u$ for all but countably many $u$.  
 (In fact, in the second case this holds for all nonzero $u$.)
\end{lemma}

\begin{remark}
Here and at various other points in the paper the assumption that $X$ be closed can be dropped by working with cohomology with compact supports instead of ordinary cohomology; we leave it to the reader to develop this.
\end{remark}

\begin{proof} Let $\gamma=\sum_{i<j}C_{ij}dx^i\wedge dx^j\in \Omega^2(V)$ be a constant coefficient $2$-form representing $-e\in H^2(V;\mathbb{R})$ where $e$ is the Euler class.  Chern-Weil theory shows that we can choose $\alpha\in \Omega^1(P)$ such that $2\pi\sqrt{-1}\alpha\in \Omega(Y;\sqrt{-1}\mathbb{R})$ is a connection form on $Y$ and such that $d\alpha=p^*\gamma$.  In particular, if $v$ is the period-$1$ vector field on $P$ that generates the $S^1$-action we have $\alpha(v)=1$.    

Since $Y\subset X$ is a coisotropic submanifold whose characteristic foliation is generated by $v$  (by the assumption that $Y$ is $\omega_V$-embedded), it follows from a standard local uniqueness theorem for coisotropic submanifolds \cite{Got} \cite{Ma} that a neighborhood $U$ of $Y$ in $X$ is symplectomorphic to $(-\ep_0,\ep_0)\times Y$ with the symplectic form \[ \omega_0=p^*\omega_V+d(s\alpha),\] where $s$ is the coordinate on $(-\ep_0,\ep_0)$.  Here and below we identify $U$ with $(-\ep_0,\ep_0)\times Y$, and we identify forms on $Y$ with their pullbacks to $U$ via the projection.

Let $\phi\in \Omega^1(V)$ be a de Rham representative with constant coefficients of any class in the image of $p_{!}\co H^2(Y;\mathbb{R})\to H^1(V;\mathbb{R})$.  Then $[\phi]\cup e=0$ by the Gysin sequence, so that $\gamma\wedge\phi=0$ since both $\phi$ and $\gamma$ have constant coefficients and their product is cohomologically trivial.  Consequently  \[ d(\alpha\wedge p^*\phi)=(d\alpha)\wedge p^*\phi =p^*(\gamma\wedge\phi)=0.\]  Thus $\alpha\wedge p^*\phi$ is a de Rham representative of a class in $H^2(Y;\mathbb{R})$; clearly we have $p_{!}(\alpha\wedge p^*\phi)=[\phi]\in H^1(V;\mathbb{R})$.  Moreover, the Gysin sequence shows 
that any class $x\in H^2(Y;\mathbb{R})$ with $p_{!}x=[\phi]$ has representatives of the form $\alpha\wedge p^*\phi+p^*\beta$, where $\beta\in \Omega^2(V)$ has constant coefficients.

Now suppose that $\phi\in \Omega^1(V)$ is a constant-coefficient $1$-form whose cohomology class lies in the image of $p_{!}\circ i^*\co H^2(X;\mathbb{R})\to H^1(V;\R)$.  Thus, for some constant-coefficient $2$-form $\beta\in \Omega^2(V)$, the cohomology class of the $2$-form \[ \eta_0=\alpha\wedge p^*\phi+ p^*\beta \] lies in the image of $i^*\co H^2(X;\mathbb{R})\to H^2(Y;\mathbb{R})$.  So there is $\zeta\in \Omega^2(X)$ and $\theta\in \Omega^1(U)$ such that $\eta_0=\zeta|_U+d\theta$.  Let $\chi\co (-\ep_0,\ep_0)\to\mathbb{R}$ be a compactly supported smooth function equal to $1$ on $(-\ep,\ep)$ for some $\ep<\ep_0$.  Then the two-form \[ \eta=\zeta+d(\chi(s)\theta) \] is a closed $2$-form on $X$ which restricts to $(-\ep,\ep)\times Y$ as (the pullback of) $\eta_0$.

Now, for $u\in\mathbb{R}$, let $\omega_u=\omega+u\eta$.  This form is closed, and since $X$ is compact and nondegeneracy is an open condition it is symplectic for sufficiently small $|u|$.  Restricting to $(-\ep,\ep)\times Y$, we have \begin{align*} \omega_u&=p^*\omega_V+d(s\alpha)+u\eta_0=p^*\omega_V+u\alpha\wedge p^*\phi+up^*\beta+sp^*\gamma+ds\wedge \alpha.
\end{align*}
Thus \[ \omega_u|_{\{s\}\times Y}=p^*(\omega_V+u\beta+s\gamma)+u\alpha\wedge p^*\phi.\] (In the case where $p\co Y\to V$ is trivial, we have $\gamma=0$.)

Now we can write $\omega_V+u\beta+s\gamma=p^*\left(\frac{1}{2}\sum_{i,j}N(u,s)dx^i\wedge dx^j\right)$ for some invertible, antisymmetric matrix $N(u,s)$ depending in affine fashion on $u$ and $s$ (and independent of $s$ in the case that $p\co Y\to V$ is trivial), and we can write $\phi=\sum \phi_idx^i$ for some constant vector $\vec{\phi}=(\phi_1,\ldots,\phi_{2n-2})$.  Associate to any $\vec{c}=(c_1,\ldots,c_{2n-2})$ the vector field $y_{\vec{c}}$ on $Y$ which is horizontal with respect to the connection (\emph{i.e.}, $\alpha(y_{\vec{c}})=0$) and projects to $V$ as the constant-coefficient vector field $\sum_{i=1}^{2n-2}c_i\partial_{x_i}$, and as before let $v$ be the vector field on $Y$ which generates the $S^1$ action and has period $1$.  Then \[ \iota_{y_{\vec{c}}}(\omega_u|_{\{s\}\times Y})=-p^*\left(\sum_{i=1}^{2n-2}(N(u,s)\vec{c})_{i}dx^i\right)-\left(\sum_{i=1}^{2n-2}u\phi_ic_i\right)\alpha\] and \[ \iota_{v}(\omega_u|_{\{s\}\times Y})=up^*\left(\sum_{i=1}^{2n-2}\phi_idx^i\right).\]  
Consequently the kernel of $\omega_u|_{\{s\}\times Y}$ is generated by $y_{\vec{c}}+v$ where $\vec{c}$ is the solution to $N(u,s)\vec{c}=u\vec{\phi}$.  (Since $N(u,s)$ is antisymmetric, this solution $\vec{c}$ is orthogonal to $\vec{\phi}$).

In the case that $p\co Y\to V$ is trivial, $v$ is just the coordinate vector field on the fiber (identified with $\mathbb{R}/\mathbb{Z}$), and so a vector field of form $y_{\vec{c}}+v$ will have no periodic orbits as long as one of the components of $\vec{c}$ is irrational.  In our case, we have $\vec{c}=uN(u)^{-1}\vec{\phi}$ (as mentioned earlier, $N$ is independent of $s$ when the bundle is trivial); as long as $\vec{\phi}\neq 0$ (which can be arranged by the assumption that $p_{!}\circ i^*$ is nontrivial) this vector $\vec{c}$ will have an irrational component for all but countably many values of $u$ near zero.  This proves the result in the case that $p\co Y\to V$ is trivial.

If instead $V$ is two-dimensional, then a vector field of the form $y_{\vec{c}}+v$ will have no periodic orbits provided that the two components of $\vec{c}$ are rationally independent (since flowlines of the vector field will project to flowlines of $\sum c_i\partial_{x_i}$ on the torus $V$).  Moreover, since $V$ is two-dimensional (so that up to scaling there is only one constant-coefficient two-form on $V$), the matrix $N(u,s)$ can be taken to have the form $n(u,s)J$ where $J=\left(\begin{array}{cc}0 &-1 \\1& 0\end{array}\right)$ for some function $n$ valued in the positive reals.  Thus for any given $u$ we have $\vec{c}=-\frac{u}{n(u,s)}J\vec{\phi}$.  Now the assumption that $p_{!}\circ i^*$ is surjective implies that we can take $\vec{\phi}$ to have rationally independent components, and in this case the vector $\vec{c}$ will, for every value of $s$, have rationally independent components for all nonzero values of $u$. 
\end{proof}

\section{Symplectic surgery}\label{surg}

Fix throughout this section any closed connected symplectic $(2n-2)$-dimensional manifold $(V,\omega_V)$.  We interpret here the symplectic cut (adapted from \cite{Le}) and the symplectic sum (introduced in \cite{Gom},\cite{MW}) as operations which relate the following two types of data: \begin{itemize}
\item[(A)] A symplectic $2n$-manifold $(X,\omega)$, a principal $S^1$-bundle $p\co Y\to V$ and an $\omega_V$-embedding $i\co Y\to X$ (thus $p^*\omega_V=i^*\omega$).
\item[(B)] A symplectic $2n$-manifold $(M,\omega')$, embeddings $i_+,i_-\co V\to M$ having disjoint images with $i_{\pm}^{*}\omega'=\omega_V$, and an isotopy class of orientation-reversing isomorphisms $\Phi\co i_{+}^{*}(Ti_+(V))^{\omega'}\to i_{-}^{*}(Ti_-(V))^{\omega'}$ of the pullbacks of the symplectic normal bundles to the $i_{\pm}(V)$ which covers the identity on $V$.\end{itemize}

More specifically, the symplectic cut takes as its input data of type (A) above and produces an output (``the cut of $X$ along $Y$'') of type (B), whereas the symplectic sum acts in the opposite direction.  As will follow from the discussion below, these operations are inverses to each other up to the obvious notion of isomorphism (an isomorphism being a symplectomorphism of the ambient manifolds which  intertwines the appropriate auxiliary data).  

These operations do not generally preserve the number of connected components of the manifold; rather, the cut of $X$ along $Y$ will have as many connected components as has $X\setminus Y$.  In our examples the manifold $M$ will often decompose into connected components as $M=M_-\coprod M_+$ with $i_-(V)\subset M_-$ and $i_+(V)\subset M_+$.  In this case the symplectic sum $X$ will sometimes be referred to as ``the symplectic sum of $M_+$ and $M_-$ along $i_+(V)$ and $i_-(V)$.''  (As indicated above, this is a slight abuse of terminology, as strictly speaking one needs more data that just $M$ and the submanifolds $i_{\pm}(V)$ to specify the sum.)

\subsection{Cuts} \label{cutsect}
We begin by describing the symplectic cut.  Let $Y\subset X$ be an $\omega_V$-embedded principal $S^1$-bundle over $V$.  Choose a $1$-form $\alpha\in\Omega^1(Y)$ so that $2\pi\sqrt{-1}\alpha$ is a connection form; thus $\alpha(v)=1$ where $v$ is the period-$1$ vector field generating the $S^1$ action.  As noted in the proof of Lemma \ref{irrat}, a neighborhood $(U,\omega|_U)$ of $Y$ in $X$ may be identified with $(-\ep,\ep)\times Y$ with the symplectic form $p^*\omega_V+d(s\alpha)$ where $s$ is the coordinate on $(-\ep,\ep)$.  Equip $U\times \mathbb{C}$ with the standard symplectic form $\tilde{\omega}=\omega|_U+\frac{\sqrt{-1}}{2}dz\wedge d\bar{z}$, and define \begin{align*} H_{\pm}\co U\times \mathbb{C}&\to \mathbb{R} \\ (s,y,z)&\mapsto s\mp\pi|z|^2.\end{align*} 
Using the convention that the Hamiltonian vector field $X_{H_{\pm}}$ is given by $\iota_{X_{H_{\pm}}}\tilde{\omega}=-dH_{\pm}$, the Hamiltonian flows of $H_{\pm}$ are given by \[ \phi_{H_{\pm}}^{t}(s,y,z)=(s,e^{2\pi it}y,e^{\mp 2\pi it}z). \]  In particular $H_+$ and $H_-$ both generate $(\mathbb{R}/\mathbb{Z})$-actions.

Let $(N_{\pm},\bar{\omega}_{\pm})$ be the reduced spaces associated to these actions at the value zero. Thus \[ N_{\pm}=\frac{H_{\pm}^{-1}(0)}{\mathbb{R}/\mathbb{Z}}\] and the symplectic form $\bar{\omega}_{\pm}$ is characterized by the fact that its pullback to $H_{\pm}^{-1}(0)$ under the quotient map coincides with the restriction of $\tilde{\omega}$ to $H_{\pm}^{-1}(0)$.

Now we have \[ H_{+}^{-1}(0)=\left\{(s,y,z)\left|s=\pi|z|^2\right.\right\};\] consequently $N_+=N_{+}^{0}\cup N_>$ where $N_{+}^{0}$ is the quotient of the subset of $H_{+}^{-1}(0)$ where $s=z=0$ and $N_>$ is the quotient of the subset where $s>0$.  Moreover the map $\phi_+\co [s,y,\sqrt{s/\pi}e^{i\theta}]\mapsto (s,e^{i\theta}y)$ defines a symplectomorphism from $N_>$ to $(0,\ep)\times Y$. Meanwhile the restriction of the projection $H_{+}^{-1}(0)\to N_+$ over $N_{+}^{0}$ is just the quotient map $p$ associated to the $S^1$-action on $Y$ (with quotient space $V$); thus we have an embedding $i_+\co V\to N_+$ with image $N_{+}^{0}$.  Using the fact that $\tilde{\omega}|_{\{0\}\times Y}=p^*\omega_V$ it's easy to see that $i_{+}^{*}\bar{\omega}_+=\omega_V$.

Similar statements apply to $N_-$: we have $N_-=N_{-}^{0}\cup N_{<}$ where $N_{<}$ is symplectomorphic to $(-\ep,0)\times Y$ via the map $\phi_-\co [s,y,\sqrt{-s/\pi}e^{i\theta}]\mapsto (s,e^{-i\theta}y)$, and where $N_{-}^{0}$ is the image of an embedding $i_-\co V\to N_-$ with $i_{-}^{*}\bar{\omega}_-=\omega_V$.

We can then form the symplectic cut $(M,\omega')$ by gluing $(X\setminus (\{0\}\times Y),\omega)$ along $(0,\ep)\times Y$ to $N_+$ via the diffeomorphism $\phi_+$, and along $(-\ep,0)\times Y$ to $N_-$ via the diffeomorphism $\phi_-$.  Since the $\phi_{\pm}$ are symplectomorphisms to their images, the symplectic forms $\omega,\bar{\omega}_+$, and $\bar{\omega}_-$ piece together to give a symplectic form $\omega'$ on $M$; the embeddings $i_{\pm}\co V\to N_{\pm}$ survive as embeddings into $M$ with $i_{\pm}^{*}\omega'=i_{\pm}^{*}\bar{\omega}_{\pm}=\omega_V$.

It remains to define an orientation-reversing isomorphism of the symplectic normal bundles to $i_{\pm}(V)$ in $M$.  Note that we have tubular neighborhood maps $N_{\pm}\to V$ sending $[s,y,z]\mapsto p(y)\in V$, canonically identifying neighborhoods $N_{\pm}$ of $i_{\pm}(V)$ with disc bundles over $i_{\pm}(V)$, in a way which is easily seen to be compatible with the orientations on the symplectic normal bundles.  Now the map $N_+\to N_-$ defined by $[s,y,z]\mapsto [-s,y,\bar{z}]$ covers the identity on $V$ and acts orientation-reversingly on the disc fibers of the maps $N_{\pm}\to V$; hence linearizing at $i_{\pm}(V)$ and projecting gives the required orientation-reversing isomorphism of symplectic normal bundles.

\subsection{Symplectic structures on Hermitian bundles}\label{bundlestruct}

If $\pi\co E\to V$ is a complex vector bundle with a Hermitian metric $\langle\cdot,\cdot\rangle_E$ and if $A$ is a unitary connection on $E$, we can define a symplectic form on a neighborhood of the zero section in the total space of $E$ as follows.  Choose a unitary connection on $E$, determining a splitting $TE\cong T^{hor}E\oplus T^{vt}E$.  Define $\sigma\co E\to\mathbb{R}$ by $\sigma(e)=-\frac{1}{4}\langle e,e\rangle_E$, and $\theta_E\in\Omega^1(E)$ by $\theta_E(v)=d\sigma(\sqrt{-1}v^{vt})$, where $v^{vt}$ is the vertical component of $v\in TE$ with respect to the connection.  Now define $\omega_E=\pi^*\omega_V+d\theta_E$.  This form is symplectic on a neighborhood of the zero section; restricts to the zero section as $\omega_V$; and restricts to the fibers of $E\to V$ as the standard symplectic form arising from the imaginary part of the Hermitian metric.

Let us connect this back to the symplectic cutting construction. As noted in the last paragraph of Section \ref{cutsect} we have a disc bundle map $N_+\to V$ defined by $[s,y,z]\mapsto p(y)$.  The subset $N_+\subset M$ can be identified as a neighborhood of the zero section in the complex line bundle $L\to V$ associated to the principal $S^1$-bundle $p\co Y\to V$ via the standard multiplicative action of $S^1$ on $\mathbb{C}$.  The standard Hermitian inner product on $\mathbb{C}$ induces a Hermitian metric on $L$, and the connection form that was chosen on $Y\to V$ in the symplectic cut construction induces a unitary connection on $L$.  It is not difficult to verify that the symplectic form induced on $N_+$ via the symplectic cut construction coincides with the symplectic form constructed from the metric and connection in the previous paragraph.

\subsection{Sums}

We now describe the symplectic sum; there are various formulations in the literature of which ours is most similar to \cite[Section 2]{IP}; in particular a reader familiar with Ionel-Parker's construction should be able to prove it to be symplectomorphic to ours without much difficulty.  Our description is designed to make the relationship of the symplectic sum with the symplectic cut as transparent as possible, and in particular to clarify that the result of the symplectic sum includes a distinguished $\omega_V$-embedded hypersurface.  A description rather similar to ours is also briefly sketched in \cite{MSy}.

We take as input a symplectic manifold $(M,\omega')$, embeddings $i_{\pm}\co V\to M$ with disjoint images and with $i_{\pm}^{*}\omega'=\omega_V$, and an orientation-reversing bundle isomorphism $\Phi\co i_{+}^{*}(Ti_+(V))^{\omega'}\to i_{-}^{*}(Ti_-(V))^{\omega'}$ covering the identity on $V$.  Write $L_+=i_{+}^{*}(Ti_+(V))^{\omega'}$.  Endow $L_+$ with an orientation-compatible complex structure, a Hermitian metric $\langle\cdot,\cdot\rangle_{L_+}$ with associated norm $|\cdot|_{L_+}$, and a unitary connection.  Then the Weinstein neighborhood theorem shows that, for some $\ep>0$, the disc bundle $L_+(\ep)=\{x\in L_+|\,|x|_{L_+}< \ep\}$ in $L_+$, equipped with the symplectic form $\omega_{L_+}=\pi_{L_+}^{*}\omega_V+d\theta_L$ of Section \ref{bundlestruct}, is symplectomorphic to a neighborhood $\mathcal{N}_+$ of $i_+(V)$ in $(M,\omega')$, via a map $\Psi_+\co L_+(\ep)\to \mathcal{N}_+$ restricting as $i_+$ on the zero-section $V$.  We take the neighborhood $\mathcal{N}_+$ to be disjoint from $i_-(V)$.

Where $Y\subset L_+$ consists of those $x\in L_+$ with $|x|_{L_+}=1$, the multiplicative action of $S^1\subset \mathbb{C}$ gives $p=\pi_{L_+}|_Y\co Y\to V$ the structure of a principal $S^1$-bundle.  The unitary connection determines a $1$-form $\alpha\in \Omega^1(V)$ characterized by the properties that it restricts as zero to the horizontal subspace of $TY\subset TL|_Y$ and gives the value $1$ to the vector field $v$ that generates the $S^1$-action and has period $1$.  Let $L_{+}^{0}$ denote the complement of the zero section in $L_+$ and let $r\co L_+\to Y$ denote the fiberwise retraction $x\mapsto \frac{x}{|x|_{L_+}}$.  Then, where we define $\rho\co L_+\to\mathbb{R}$ by  $\rho(x)=\pi|x|_{L_+}^{2}$, an easy computation shows \[ \omega_{L_+}|_{L_{+}^{0}}=\pi_{L_+}^{*}\omega_V+d(\rho r^*\alpha)=r^*p^*\omega_V+d(\rho r^*\alpha).\]  Thus the map $f_+\co x\mapsto (\pi|x|_{L_+}^{2},r(x))$ defines a symplectomorphism from $L_{+}^{0}(\ep)=\{x\in L_+|0<|x|_{L_+}< \ep\}$ to its image $(0,\pi\ep^2)\times Y$, where the latter is equipped with the symplectic form $p^*\omega_V+d(s\alpha)$ with $s$ being the coordinate on the interval $(0,\pi\ep^2)$.

 Write $L_-=i_{-}^{*}(Ti_-(V))^{\omega'}$; thus we have an orientation-reversing bundle isomorphism $\Phi\co L_+\to L_-$ such that $\pi_{L_-}\circ\Phi=\pi_{L_+}$.  Endow $L_-$ with the unique complex structure which makes $\Phi$ conjugate-linear (thus this complex structure is compatible with the orientation on $L_-$), and with the Hermitian metric $\langle\cdot,\cdot\rangle_{L_-}$ given by \begin{equation} \label{phiherm}\langle\Phi v,\Phi w\rangle_{L_-}=\langle w,v\rangle_{L_+}.\end{equation}    Pushing forward the horizontal subspaces of $TL_+$ via $\Phi$ gives a connection on $L_-$, which is unitary with respect to $\langle \cdot,\cdot\rangle_{L_-}$.  These data allow us to construct a symplectic form $\omega_{L_-}=\pi_{L_-}^{*}\omega+d\theta_{L_-}$ on a neighborhood of the zero section in $L_-$ via the procedure of Section \ref{bundlestruct}.  The Weinstein neighborhood theorem shows that, for some $\delta>0$, the $\delta$-disc bundle $L_-(\delta)$ in $L_-$ is symplectomorphic to a neighborhood of $\mathcal{N}_-$ of $i_-(V)$ in $(M,\omega')$ via a map $\Psi_-\co L_-(\delta)\to \mathcal{N}_-$ restricting as $i_-$ on $V$.  Here we choose $\mathcal{N}_-$ so that $\mathcal{N}_-\cap \mathcal{N}_+=\varnothing$.
 
Now by (\ref{phiherm}), the image $\Phi(Y)$ will be the unit circle bundle in $L_-$.  Then as usual one obtains a $1$-form $\alpha_-\in \Omega^1(Y)$ by requiring $\alpha_-$ to vanish on horizontal vectors and evaluate as $1$ on the period-$1$ generator of the $S^1$-action.  As before, where $\rho_-\co L_-\to \mathbb{R}$ is defined by $\rho_-(y)=\pi|y|_{L_-}^{2}$, and where (for $L_{-}^{0}$ equal to the complement of the zero-section) $r_-\co L_{-}^{0}\to \Phi(Y)$ is defined by $r_-(y)=\frac{y}{|y|_{L_-}}$, we have \[ (\omega_{L_-})|_{L_{-}^{0}}=\pi_{L_-}^{*}\omega_V+d(\rho_-r_{-}^{*}\alpha_-).\]  Now clearly $r_-\circ \Phi=r$, and $\rho_-\circ\Phi=\rho$; meanwhile since $\Phi$ is conjugate-linear we will have $\Phi^*\alpha_-=-\alpha$.  So since $\Phi^*\pi_{L_-}^{*}\omega_V=\pi_{L_+}^{*}\phi^*\omega_V=\pi_{L_+}^{*}\omega_V$ we get \[ \Phi^*\left((\omega_{L_-})|_{L_{-}^{0}}\right)=\pi_{L_+}^{*}\omega_V-d(\rho r^*\alpha).\]
Consequently the map $f_-\co y\mapsto \left(-\pi|y|_{L_-}^{2},r(\Phi^{-1}(y))\right)$ is a symplectomorphism from the complement  $L_{-}^{0}(\delta)$ of the zero section in the $\delta$-disc bundle $L_{-}(\delta)$ to its image $(-\pi\delta^2,0)\times Y $ endowed with the symplectic form $p^*\omega_V+d(s\alpha)$.

We now define the symplectic sum $(X,\omega)$.  It is the union of the symplectic manifold $(M\setminus (i_-(V)\cup i_+(V)),\omega')$ with the symplectic manifold $((-\pi\delta^2,\pi\ep^2)\times Y,p^*\omega_V+d(s\alpha))$, where $(-\pi\delta^2,0)\times Y$ is glued to the Weinstein neighborhood $\mathcal{N}_-\setminus i_-(V)$ by the symplectomorphism $\Psi_-\circ f_{-}^{-1}$, and $(0,\pi\ep^2)\times Y$ is glued to the Weinstein neighborhood $\mathcal{N}_+\setminus i_+(V)$ by the symplectomorphism $\Psi_{+}\circ f_{+}^{-1}$.  In $(X,\omega)$ we see a distinguished hypersurface $Y\cong \{0\}\times Y$; this hypersurface carries the structure of a principal $S^1$-bundle, and evidently $\omega|_Y=p^*\omega_V$.

It is basically immediate from the descriptions that we have given that the cutting and summing constructions are inverse to each other up to symplectomorphisms that preserve all of the original data.  Indeed, if we start with $(M,\omega')$ together with $i_{\pm}\co V\to M$ and $\Phi\co L_+\to L_-$ and then apply the sum, as just noted we obtain a hypersurface $Y\subset X$ with a neighborhood on which the symplectic structure appears as $p^*\omega_V+d(s\alpha)$ just as in the cut construction.  Applying the cut to this hypersurface replaces this neighborhood with a union of two disc bundles over symplectically embedded copies of $V$, with connections induced by the $1$-form $\alpha$, and as noted in Section \ref{bundlestruct} the symplectic forms on these disc bundles induced by the cut construction coincide with the symplectic forms that we began with as Weinstein-type models of the neighborhoods of $i_{\pm}(V)$.  Conversely, if we start with $Y\subset (X,\omega)$ and apply the cut followed by the sum, we have seen that, first, a neighborhood $(-\pi\ep^{2},\pi\ep^2)\times Y$ with the symplectic structure $p^*\omega_V+d(s\alpha)$  gets replaced by two $\ep$-disc bundles, and then upon summing these disc bundles get replaced by the original  $(-\pi\ep^2,\pi\ep^2)\times Y$.  

\begin{center}\begin{figure}
\includegraphics[scale=0.5]{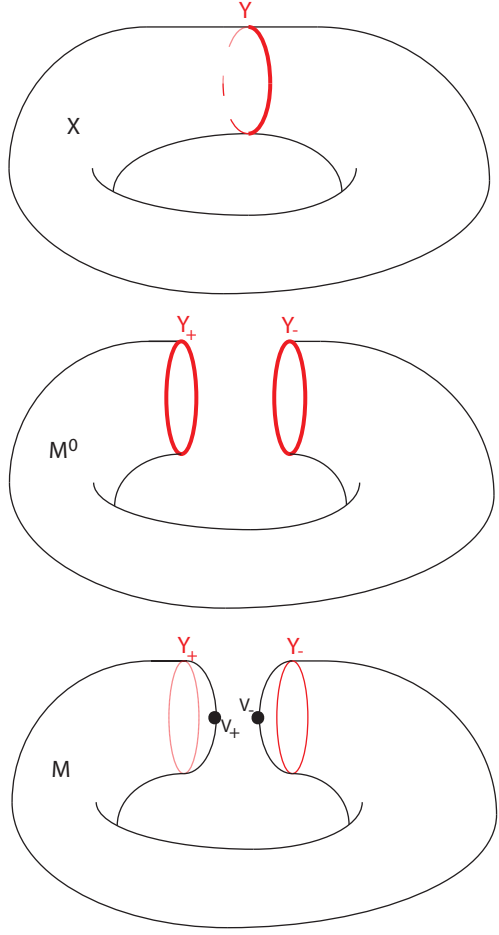} 
\caption{The manifolds $X$, $M^0$, and $M$ in a case where $2n=2$, where $M$ is the symplectic cut of $X$ along $Y$ (and, conversely, $X$ is obtained from $M$ by the symplectic sum).  We write $V_{\pm}=i_{\pm}(V)$. Contrary to what the picture suggests, $X$ and $M$ have equal volume.}
\end{figure}
\end{center}

\section{The main cohomology calculation}\label{calc}

Consider as before a closed  oriented $(2n-2)$-dimensional manifold $V$ (the symplectic structures won't really be relevant in this section), a principal $S^1$-bundle $p\co Y\to V$, and an inclusion $Y\to X$ as a hypersurface in an oriented manifold $X$, with a tubular neighborhood identified with $[-\ep,\ep]\times Y$  (even without symplectic structures, $Y$ obtains an orientation by the ``fiber-first'' prescription for $p\co Y\to V$, so since $Y$ is a closed oriented hypersurface of an oriented manifold such a tubular neighborhood exists, with the orientation of $TX|_Y$ given by the orientation on $[-\ep,\ep]$ followed by that on $Y$).  

Let us first re-describe the topological operation underlying the (symplectic) cut of $X$ along $Y$. Let $M^0=X\setminus (-\ep,\ep)\times Y$; thus where $Y_{\pm}=\{\pm\ep\}\times Y$ we have, as oriented manifolds, $\partial M^0=(-Y_+)\coprod Y_-$, where we take the orientations on $Y_{\pm}$ to both be the same as the orientations on $Y$.  Then where $\mathcal{N}_{\pm}$ are the associated disc bundles \[ \mathcal{N}_{\pm}=\frac{Y_+\times D^2}{(e^{i\theta}y,z)\sim (y,e^{\pm i\theta}z)} \] over $V$, we have canonical identifications $\partial \mathcal{N}_{\pm}=\pm Y_{\pm}$.  The zero-sections of these disc bundles are the images of canonical inclusions $i_{\pm}\co V\to \mathcal{N}_{\pm}$

Denote by $\mathcal{N}$ and $\tilde{Y}$ the \emph{oriented} manifolds $\mathcal{N}=\mathcal{N}_+\coprod \mathcal{N}_-$ and $\tilde{Y}=\partial\mathcal{N}$.  Then the cut of $X$ along $Y$ is, as a smooth oriented manifold, \[ M=\mathcal{N}\cup_{\tilde{Y}}M^0, \] \emph{i.e.}, $M$ is the manifold obtained by gluing $\mathcal{N}$ and $M^0$ along $\tilde{Y}=\partial \mathcal{N}=-\partial M^0$.

Below we will use $i$ as a generic symbol for inclusion maps.  We have a Mayer-Vietoris sequence \begin{equation}\label{mv1} \xymatrix{ H^2(M;\mathbb{R})\ar[r]^<<<<{\left(\begin{array}{c}i^* \\ i^*\end{array}\right)} & H^2(\mathcal{N};\mathbb{R})\oplus H^2(M^0;\mathbb{R})\ar[r]^>>>>{(i^*,-i^*)} & H^2(\tilde{Y};\mathbb{R})\ar[r]^{\delta} & H^3(M;\mathbb{R}) }\end{equation}

Now $\mathcal{N}_{\pm}$ deformation retracts onto $V$ via maps $r_{\pm}$ such that, where $i\co Y_{\pm}\to \mathcal{N}_{\pm}$ is the inclusion and $p\co Y_{\pm}\to V$ is the projection (obtained by the canonical identification of $Y_{\pm}$ with $Y$), we have $r_{\pm}\circ i=p$.  Thus we have identifications $H^2(\mathcal{N};\mathbb{R})\cong H^2(V;\mathbb{R})\oplus H^2(V;\R)$ and $H^2(\tilde{Y};\R)=H^2(Y;\R)\oplus H^2(Y;\R)$, under which the inclusion induced map $i^*\co H^2(\mathcal{N};\R)\to H^2(\tilde{Y};\R)$ becomes $\left(\begin{array}{cc} p^* & 0\\ 0 & p^*\end{array}\right)$.  Thus the Mayer-Vietoris sequence (\ref{mv1}) becomes
\begin{equation} \label{mv2} \xymatrix{
H^2(M;\R)\ar[r]^<<<<{\left(\begin{array}{c}i_{+}^{*}\\i_{-}^{*}\\i^*\end{array}\right)} & H^2(V;\R)\oplus H^2(V;\R)\oplus H^2(M^0;\R) \ar[r]^>>>>>>{A} & H^2(Y;\R)\oplus H^2(Y;\R) \ar[d]^{\delta} \\ & &H^3(M;\R)
}\end{equation} 
where $A\co H^2(V;\R)\oplus H^2(V;\R)\oplus H^2(M^0;\R)\to H^2(Y;\R)\oplus H^2(Y;\R)$ is the map \[ A(v_+,v_-,m)=(p^*v_+ -m|_{Y_+},p^*v_- -m|_{Y_-}).\]

Meanwhile, it is quite generally the case that if $W$ is a closed oriented $m$-manifold formed from oriented manifolds $W_1$ and $W_2$ with closed boundary $Z=\partial W_1=-\partial W_2$ as $W=W_1\cup_{Z}W_2$, then the connecting homomorphism $\delta\co H^{k-1}(Z;\R)\to H^k(W;\R)$ is conjugated by Poincar\'e duality to the inclusion-induced map $i_*\co H_{m-k}(Z;\R)\to H_{m-k}(W;\R)$.  Applying this to our situation (and assuming $M$ to be closed), we have an inclusion \[ \tilde{i}=\tilde{i}_+\coprod \tilde{i}_-\co \tilde{Y}=Y_+\coprod (-Y_-)\to M.\]  Recall our convention that the oriented manifolds $Y_{\pm}=\{\pm\ep\}\times Y$ have the orientations directly inherited from $Y$.  So since reversing the orientation of a manifold reverses the sign of Poincar\'e duality, we have a commutative diagram 
\begin{equation} \label{PDY} \xymatrix{ H^2(Y_+;\R)\oplus H^2(Y_-;\R) \ar[r]^<<<<<<<<{\delta} \ar[dd]_{\left(\begin{array}{cc}PD_Y & 0\\ 0 & PD_Y\end{array}\right)} & H^3(M;\R)\ar[dd]^{PD_M} \\ &  \\  H_{2n-3}(Y_+;\R)\oplus H_{2n-3}(Y_-;\R) \ar[r]^>>>>>{(\tilde{i}_{+\ast}-\tilde{i}_{-\ast})} & H_{2n-3}(M;\R) }
\end{equation} where $\tilde{i}_{\pm}$ are the inclusions of $Y_{\pm}$ into $M$.

Meanwhile, there is an obvious map of pairs $\pi\co (M^0,\partial M^0)\to (X,Y)$ which collapses the two boundary components $Y_{\pm}=\{\pm\ep\}\times Y$ to $Y$.  The cohomology exact sequences of the pairs $(X,Y)$ and $(M^0,\partial M_0)=(M^0,\tilde{Y})$ give a commutative diagram with exact rows  
\begin{equation}\label{reles} \xymatrix{
H^2(X;\R)\ar[d] \ar[r] & H^2(Y;\R) \ar[d]\ar[r] & H^3(X,Y;\mathbb{R}) \ar[d] \\
H^2(M^0;\R)\ar[r] & H^2(\tilde{Y};\R) \ar[r] & H^3(M^0,\tilde{Y};\R) 
} \end{equation}

Now it is easy to see by an excision argument that the maps $\pi^*\co H^*(X,Y;\R)\to H^*(M^0,\partial M^0;\R)$ are isomorphisms; in particular the rightmost vertical arrow in the above diagram is an isomorphism.  Meanwhile, under the obvious identifications $H^2(\tilde{Y};\R)\cong H^2(Y_+;\R)\oplus H^2(Y_-;\R)=H^2(Y;\R)\oplus H^2(Y;\R)$ the middle vertical arrow in (\ref{reles}) is the map $a\mapsto (a,a)$.  Consequently the exactness of the rows shows that \begin{align}\label{m0xy} Im &\left(i^*\co H^2(X;\R) \to H^2(Y;\R)\right)\nonumber \\ 
&=
\left\{a\in H^2(Y;\R)\left| (a,a)\in Im\left(\left(\begin{array}{cc}i^* \\ i^*\end{array}\right)\co H^2(M^0)\to H^2(Y_+;\R)\oplus H^2(Y_-;\R) \right) \right.\right\}.\end{align}

We can now prove the main lemma:

\begin{lemma}\label{maincalc}  Assume that $X$ and hence $M$ are closed and that either \begin{itemize} \item[(i)] The inclusion-induced maps $H^2(M^0;\R)\to H^2(Y_+;\R)$ and $H^2(M^0;\R)\to H^2(Y_-;\R)$ are equal (with respect to the canonical identifications $H^2(Y_{\pm};\R)=H^2(Y;\R)$); \textbf{or}
\item[(ii)] The map $i_{+\ast}+i_{-\ast}\co H_2(V;\R)\oplus H_2(V;\R)\to H_2(M;\R)$ is injective.\end{itemize}
Then where $p_{!}\co H^2(Y;\R)\to H^1(V;\R)$ is the Gysin map\begin{align}\label{mainequality} Im(p_!\circ i^* &\co H^2(X;\R)  \to H^1(V;\R))\nonumber \\=&(Im\, p_!)\cap PD_V\left(\ker\left((i_+)_*-(i_-)_*\co H_{2n-3}(V;\R)\to H_{2n-3}(M;\R)\right)\right).\end{align}
\end{lemma}

\begin{cor}\label{4cor} If $2n=4$ and if the surfaces $i_+(V)$ and $i_-(V)$ are homologically linearly independent in $M$, then \begin{equation}\label{maindim4}
Im(p_!\circ i^*\co H^2(X;\R)\to H^1(V;\R))=
 PD_V\left(\ker\left((i_+)_*-(i_-)_*\co H_{1}(V;\R)\to H_{1}(M;\R)\right)\right)\end{equation}
\end{cor}

Indeed, in this case $H_2(V;\R)$ has rank one and so the linear independence of the $i_{\pm}(V)$ is equivalent to condition (ii) in Lemma \ref{maincalc}; moreover since $H^3(V;\R)=0$ when $2+\dim V=2n=4$ the Gysin sequence (\ref{gysin}) shows that $p_!$ is automatically surjective.

\begin{proof}[Proof of Lemma \ref{maincalc}] Recall that $p_!=PD_V\circ p_*\circ PD_Y$ (here we use $PD_W$ to denote Poincar\'e duality in a closed oriented manifold $W$ and its inverse, both considered as a map from homology to cohomology and vice versa; in particular $PD_W$ squares to the identity).  Now within $M$, any chain $c$ in $Y_{\pm}$ is homologous to its image $p_*c$ in $i_{\pm}(V)$.  Consequently for $\eta\in H_*(Y_{\pm};\R)$ we have $(\tilde{i}_{\pm})_*\eta=(i_{\pm})_*p_*c$.

For the forward inclusion (which the proof will show to be true even without hypotheses (i) or (ii)), suppose that $x\in Im(p_{!}\circ i^*)$; say $x=p_{!}(i^*n)$ where $n\in H^*(X;\R)$; write $a=i^*n$.  Then where $m=n|_{M^0}$ we have $(a,a)=A(0,0,-m)$, and so $\delta(a,a)=0\in H^3(M;\R)$.  Hence by the commutativity of (\ref{PDY}) and by the first paragraph of this proof, \[ (i_+)_*p_*PD_Y(a)-(i_-)_*p_*PD_Y(a)=(\tilde{i}_{+})_*PD_Y(a)-(\tilde{i}_-)_*PD_Y(a)=0.\]
So since $x=p_!a$, \emph{i.e.}, $PD_V(x)=p_*PD_Y(a)$, this shows that $x$ belongs to the right-hand side of (\ref{mainequality}), proving the forward inclusion.

As for the reverse inclusion, assume that $(i_+)_*PD_V(x)-(i_-)_*PD_V(x)=0$ and that $x=p_!a$ where $a\in H^2(Y;\R)$.  Thus $p_*PD_Y(a)=PD_V(x)$, so by the first paragraph of the proof $(\tilde{i}_+)_*PD_Y(a)-(\tilde{i}_+)_*PD_Y(a)=0$, and so by the commutativity of (\ref{PDY}) $\delta(a,a)=0$.  Thus by the exactness of (\ref{mv2}) there are $m\in H^2(M_0;\R), v_{\pm}\in H^2(V;\R)$ so that \begin{equation}\label{twoa} a=p^*v_+-m|_{Y_+}=p^*v_- - m|_{Y_-}.\end{equation}
  
If we are in case (i), then $m|_{Y_-}=m|_{Y_+}$, so that $a-p^*v_+=a-p^*v_-$, and by the exactness of (\ref{gysin}) (specifically the fact that $p_!\circ p^*=0$) we have $p_!(a-p^*v_+)=p_!a=x$.  Moreover  $(a-p^*v_+,a-p^*v_+)\in Im(H^2(M_0;\R)\to H^2(Y_+;\R)\oplus H^2(Y_-;\R))$; thus by (\ref{m0xy}) we have $a-p^*v_+\in Im(H^2(X;\R)\to H^2(Y;\R))$.  Thus \[ x=p_!(a-p^*v_+)\in Im(p_!\circ i^*),\] proving the lemma in case (i).

As for case (ii), the assumption that $(i_+)_*+(i_-)_*\co H_2(V;\R)\oplus H_2(V;\R)\to H_2(M;\R)$ is injective implies that the dual map 
$(i_{+}^{*},i_{-}^{*})\co H^2(M;\R)\to H^2(V;\R)\oplus H^2(V;\R)$ is surjective.  So where $v_+,v_-,m$ are as in (\ref{twoa}), there is $n\in H^2(M;\R)$ so that $i_{+}^{*}n=p^*v_+$ and $i_{-}^{*}n=p^*v_-$. But then $n|_{M^0}-m$ will restrict both to $Y_+$ and to $Y_-$ as $a$, and so by (\ref{m0xy}) $a\in Im(H^2(X;\R)\to H^2(Y;\R))$.  Thus $x\in Im(p_!\circ i^*)$ in this case as well.
\end{proof}

\begin{remark}\label{case1} In most examples below, we will begin not with the manifold $X$ but with manifold $M$, as a result of which interpreting Case (i) of Lemma \ref{maincalc} will sometimes be more subtle since $Y$ is not part of the initial data.  As such it is useful to re-express the condition of Case (i) so that it can be interpreted directly in terms of the initial data for the symplectic sum.  To do this, note that $Y_{\pm}$ appear as boundaries of symplectic tubular neighborhoods of $i_{\pm}(V)$; these tubular neighborhoods give rise to embeddings $I_{\pm}\co Y_{\pm}\to X\setminus(i_+(V)\cup i_-(V))$.  Meanwhile, considering the $Y_{\pm}$ as unit-circle bundles in $L_{\pm}=i_{\pm}^{*}Ti_{\pm}(V)^{\omega'}$, the map $\Phi$ restricts to a diffeomorphism $\Phi\co Y_+\to Y_-$.  In these terms, it is easy to see that the condition in Case (i) of Lemma \ref{maincalc} translates to the statement that the maps $I_{+}^{*},(I_-\circ \Phi)^{*}\co H^2(X\setminus (i_+(V)\cup i_-(V));\R)\to H^2(Y_+;\R)$ are equal. \end{remark}

\section{Examples}\label{ex}

\subsection{Examples using $M=T^{2n-2}\times S^2$}
Let us first see how the most familiar example of a manifold admitting aperiodic symplectic forms can be interpreted in our context.

\begin{ex}[The torus] \label{torus}
Let $(X,\omega)$ be the $2n$-dimensional torus, with the standard symplectic form $\omega=\sum_{i=1}^{n}dx^{2i-1}\wedge dx^{2i}$ (we take $x_i\in(\mathbb{R}/\mathbb{Z})$).  Let $(V,\omega_V)$ be the $(2n-2)$-dimensional torus, with $\omega_V=\sum_{i=1}^{n-1}dx^{2i-1}\wedge dx^{2i}$.
If \[ Y=\{(x_1,\ldots,x_{2n})\in X|x_{2n}=0\},\] we have a principal bundle $p\co Y\to V$ given by $(x_1,\ldots,x_{2n-1},0)\mapsto (x_1,\ldots,x_{2n-2})$ (with the $S^1$ action $e^{2\pi it}(x_1,\ldots,x_{2n-1},0)=(x_1,\ldots,x_{2n-1}+t,0)$ on $Y$), and $Y$ is $\omega_V$-embedded in $X$.  Consequently we may construct the symplectic cut $(M,\omega')$ of $X$ along $Y$: topologically this is equivalent to cutting $X$ along $Y$ to obtain $M^0=T^{2n-2}\times (S^1\times[0,1])$, and then collapsing $S^1\times \{0\}$ and $S^1\times \{1\}$ to points (denoted $0$ and $\infty$ below), yielding a manifold $M$ diffeomorphic to $T^{2n-2}\times S^2$.  Following through the details of the symplectic cut construction, one sees that $(M,\omega')$ is symplectomorphic to $T^{2n-2}\times S^2$ (with a split symplectic form where the $T^{2n-2}$ factor is symplectomorphic to $V$ and the $S^2$ factor has area $1$); the embeddings $i_{\pm}\co V\to M$ are given by the obvious inclusions of  $T^{2n-2}\times \{0\}$ and $T^{2n-2}\times\{\infty\}$.  The symplectic normal bundles to $i_{\pm}(V)$ are naturally identified with the trivial bundles over $T^{2n-2}$ with fibers $T_{\infty}S^2$ and $T_0 S^2$, and the orientation-reversing bundle isomorphism $\Phi\co i_{+}^{*}Ti_+(V)^{\omega'}\to i_{-}^{*}Ti_-(V)^{\omega'}$ is induced by the map $T_{\infty}S^2\to T_0S^2$ obtained by linearizing $z\mapsto \frac{1}{\bar{z}}$ (where we make the usual identification of $S^2$ with $\mathbb{C}\cup\{\infty\}$).

Conversely, we could have started with the symplectic manifold $(M,\omega')$ defined as the product of $T^{2n-2}\times S^2$ with a split symplectic structure, together with the inclusions $i_{\pm}\co V\to M$ of $T^{2n-2}\times\{0\}$ and $T^{2n-2}\times\{\infty\}$ and the map between their normal bundles obtained by linearizing $z\mapsto\frac{1}{\bar{z}}$.  Applying the symplectic sum construction to these initial data yields the torus $T^{2n}$ with its standard symplectic structure.   Up to isotopy, in the language of Remark \ref{case1} and identifying $Y_{\pm}$ with $T^{2n-2}\times S^1$, the maps 
$I_+$ and $I_-\circ \Phi$ can be taken to be, respectively \[ (v,e^{i\theta})\mapsto \left(v,\frac{1}{\ep e^{i\theta}}\right)\mbox{ and }(v,e^{i\theta})\mapsto (v,\ep e^{-i\theta}) \] for small $\ep$.  In particular $I_+$ and $I_-\circ \Phi$ are homotopic, so the induced maps $H^2(X\setminus(i_+(V)\cup i_-(V));\R)\to H^2(Y_+;\R)$ are equal and Case (i) of Lemma \ref{maincalc} applies.  Now the maps $i_+$ and $i_-$ are homotopic, so $(i_+)_*-(i_-)_*\co H_{2n-3}(V;\R)\to H_{2n-3}(M;\R)$ vanishes identically and so Lemma \ref{maincalc} shows that \[ Im(p_!\circ i^*\co H^2(X;\R)\to H^1(V;\R))=H^1(V;\R).\]

Consequently Lemma \ref{irrat} allows us to use any nonzero element $x\in H^1(V;\R)$ to perturb the symplectic form $\omega$ to forms $\omega+u\eta$ (where $p_![\eta]=x$) such that, for all but countably many $u$, the hypersurface $Y$ has a tubular neighborhood in which all of its parallel slices have no closed characteristics with respect to $\omega+u\eta$.  In particular $\omega_u$ is aperiodic for all but countably many values of $u$; indeed, $\omega_u$ can be taken to be one of the forms in Zehnder's original example \cite{Ze}.
\end{ex}

\begin{ex}[The Kodaira--Thurston manifold and variants] \label{kt}  Again let $(V,\omega_V)$ be the $(2n-2)$-torus with its standard symplectic structure.  Let $(M,\omega')$ be the sympelctic manifold $V\times S^2$ with a split symplectic structure, just as in Example \ref{torus}.   Let $\psi\in SL(2n-2,\mathbb{Z})\cap Sp(2n-2,\mathbb{R})$ be any linear symplectomorphism of $V$.  By applying the symplectic sum to $(M,\omega')$ we will construct a symplectic $2n$-manifold $(X_{\psi},\omega)$ which is equal to the standard $2n$-torus when $\psi$ is equal to the identity and for which, more generally, $\omega$ admits aperidic perturbations $\omega_u$ whenever $1$ is an eigenvalue of $\psi$.  When $2n=4$ and $\psi=\left(\begin{array}{cc}1 & 1\\0 & 1\end{array}\right)$, $X_{\psi}$ will be equal to the standard Kodaira--Thurston manifold \cite{Th}.

To do this, define symplectic embeddings $i_{\pm}\co V\to M$ by $i_+(v)=(v,\infty)$ and $i_-(v)=(\psi(v),0)$.  The pullbacks of the symplectic normal bundles $i_{\pm}^{*}Ti_{\pm}(V)^{\omega'}$ have obvious identifications as the trivial bundles over $V$ with fibers, respectively $T_{\infty}S^2$ and $T_0 S^2$, so just as in Example \ref{torus} we can let the isomorphism $\Phi$ be the map that results from linearizing $z\mapsto \frac{1}{\bar{z}}$ in these trivializations.  This results in a symplectic manifold $(X_{\psi},\omega)$ containing a distinguished hypersurface $Y$.

In the notation of Remark \ref{case1} we have $I_+(v,e^{i\theta})=(v,\ep^{-1}e^{-i\theta})$ and $I_-\circ\Phi(v,e^{i\theta})=(\psi(v),\ep e^{i\theta})$, so if $\psi$ is not the identity Case (i) of Lemma \ref{maincalc} does \emph{not} apply.  So since Case (ii) clearly doesn't apply either we cannot use Lemma \ref{maincalc} in this example.  However, we can determine $Im(p_!\circ i^*\co H^2(X_{\psi};\R)\to H^1(V;\R))$ by a more direct route.

Indeed, up to diffeomorphism preserving $Y$, $X_{\psi}$ is formed by removing tubular neighborhoods of $i_{\pm}(V)$ with boundaries $Y_{\pm}$ to create $M^0$, and then by gluing the resulting boundary components by the diffeomorphism $\Phi\co Y_+\to Y_-$ to form $X_{\psi}$, with $Y$ appearing as the image of either boundary component under the quotient $M^0\to X$.  In this case we evidently have $M^0\cong T^{2n-2}\times [0,1]\times S^1$, with $\Phi$ acting by $(v,1,e^{i\theta})\mapsto (\psi(v),0,e^{i\theta})$.  In other words,  where $T_{\psi}\to (\mathbb{R}/\mathbb{Z})$ is the mapping torus of $\psi\co V\to V$, we have $X_{\psi}=S^1\times T_{\psi}$ with $Y$ appearing as $S^1\times V\times\{0\}$ and the projection $p\co Y\to V$ just given by the projection onto the $V$ factor.  Consequently \[ Im(p_{!}\circ i^*\co H^2(X;\R)\to H^1(V;\R))= Im(H^1(T_{\psi};\R)\to H^1(V;\R)) \] where the map on the right is that induced by the inclusion of a fiber into the mapping torus $T_{\psi}$. This map fits into a well-known exact sequence \[\xymatrix{ H^0(V;\R)\ar[r]& H^1(T_{\psi};\R)\ar[r]&H^1(V;\R)\ar[r]^{1-\psi^*} & H^1(V;\R) },\] and in particular the image of the map is equal to $\ker(1-\psi^*)$, which is nontrivial by the assumption that $\psi$ has $1$ as an eigenvalue.  Thus whenever $1$ is an eigenvalue of $\psi$, Lemma \ref{irrat} provides symplectic forms $\omega_u$ on $X_{\psi}$ with respect to which $Y$ violates the nearby existence property, and so $\omega_u$ is aperiodic.
\end{ex} 

\subsection{Homotopy elliptic surfaces}\label{elliptic} \begin{ex}[The surfaces $E(n)$]\label{k3}
Let $f_0,f_1\in\mathbb{C}[x_0,x_1,x_2]$ be two generic cubic homogeneous polynomials, so that their common projective vanishing locus in $\mathbb{C}P^2$ consists of nine distinct points.  Symplectically blow up $\mathbb{C}P^2$ at these nine points to obtain a symplectic manifold $(M_+,\omega'_{+})$, which contains a self-intersection zero torus $T_+$ arising as the proper transform of the zero set of $f_0$ in $\mathbb{C}P^2$.  The manifold $M_+$ (which is sometimes denoted $E(1)$) comes equipped with a singular fibration $\pi\co M_+\to \mathbb{C}P^1$ whose fiber over $[t_0:t_1]$ is the proper transform of the vanishing locus of $t_0f_0+t_1f_1$.  In particular the symplectic normal bundle of $T_+=\pi^{-1}([1:0])$ may be identified with the trivial bundle over $T_+$ whose fiber is $T_{[1:0]}\mathbb{C}P^1$.  Let $(M_-,\omega'_{-})$ be another copy of $(M_+,\omega'_+)$, containing a symplectic torus $T_-$ corresponding to $T_+$, and let $(M,\omega')$ be the disjoint union of the symplectic manifolds $(M_{\pm},\omega'_{\pm})$.  Where $\Phi\co (TT_{+})^{\omega'}\to (TT_-)^{\omega'}$ is the orientation-reversing bundle isomorphism which acts as the identity on the factor $T_+$ of $(TT_{+})^{\omega'}=T_+\times T_{[1:0]}\mathbb{C}P^1$ and which acts by complex conjugation on the other factor, we can then form the symplectic sum $(X,\omega)$ associated to the data $(M,T_{\pm},\Phi)$.  This manifold carries a singular fibration $X\to S^2$ with generic fiber a symplectic square-zero torus; the distinguished hypersurface $Y\subset X$ associated to the symplectic sum is the preimage of the equator under this fibration.  As is well-known (see for instance \cite[Sections 3.1, 7.3]{GS}), $X$ is diffeomorphic to the $K3$ surface (which we will denote by $E(2)$ below), \emph{i.e.}, the zero locus in $\mathbb{C}P^3$ of a generic quartic.

Now the tori $T_{\pm}$ are obviously homologically linearly independent in $M$ (they are homologically nontrivial by virtue of being symplectic, and they are on different connected components), and $M$ is simply-connected, so  Corollary \ref{4cor} applies to show that the composition $p_{!}\circ i^*\co H^2(X;\R)\to H^2(Y;\R)\to H^1(T_+;\R)$ is surjective.  Hence Lemma \ref{irrat} produces symplectic forms $\omega_u$ on $E(2)$ with respect to which the hypersurface $Y$ violates the nearby existence property.  


One can iterate this construction, applying the symplectic sum first to $E(2)\coprod M_-$ along $T$ and $T_-$ to obtain the elliptic surface $E(3)$ with generic fiber a symplectic torus $T(3)$, and then for $n\geq 3$, applying the symplectic sum to $E(n)\coprod M_-$ along $T(n)$ and $T_-$ to obtain $E(n+1)$.  For all $n$, $E(n)$ and $E(n)\setminus T(n)$ are both simply connected.  Applying Lemma \ref{irrat} then produces, for all $n\geq 2$, symplectic forms on $E(n)$ such that a hypersurface $Y(n)$ disjoint from $T(n)$ violates the nearby existence property.
\end{ex}

\begin{remark} \label{othersum} Suppose that $(M_+,\omega'_+)$ is a symplectic four-manifold containing a closed hypersurface $Y$ which violates the nearby existence property, and containing additionally a closed symplectic surface $V_+$ which is disjoint from $Y$.  If $(M_-,\omega'_-)$ is any symplectic four-manifold containing a closed symplectic surface $V_-$ having the same genus and opposite self-intersection number to $V_+$, then after rescaling the symplectic form $\omega'_-$ the surfaces $(V_+,\omega'_+)$ and $(V_-,\omega'_-)$ will be symplectomorphic (since by the Moser trick the genus and area are complete invariants of closed symplectic two-manifolds), and the assumption on the intersection numbers implies that there will exist an orientation-reversing isomorphism of their symplectic normal bundles.  Thus we can apply the symplectic sum construction to the manifold $M=M_+\coprod M_-$ along $V_+$ and $V_-$ to obtain a manifold $(X,\omega)$.  (Indeed, we can do this in several ways, since there is some freedom in the choice of symplectomorphism between $V_-$ and $V_+$ and in the choice of orientation-reversing bundle isomorphism).  Since $Y\subset M_+$ is disjoint from $V_+$, the sum $(X,\omega)$ will contain a neighborhood symplectomorphic to any given sufficiently small tubular neighborhood of $Y$ in $M_+$.  Consequently the fact that $Y$ violates the nearby existence property in $(M_+,\omega_+)$ implies that it also violates this property in $(X,\omega)$.

In particular, applying this with $M_+=E(n)$ and $V_+=T(n)$, we conclude that for any $n\geq 2$, we can choose the symplectic form on $E(n)$ in such a way that any symplectic sum of $E(n)$ with another symplectic manifold along  $T(n)$ will be aperiodic.
\end{remark}

\begin{ex}[The log transforms $E(n)_{p,q}$]  If $p$ is an integer, a new elliptic surface $E_p\to S^2$ can be constructed from an old one $E\to S^2$ by the operation of a ``logarithmic transformation of order $p$'': in the smooth category this amounts to removing a neighborhood of a regular fiber and then regluing it by a diffeomorphism $\phi\co T^2\times \partial D^2\to T^2\times \partial D^2$ which sends $\{pt\}\times \partial D^2$ to a loop whose homology class projects to $H_1(\{pt\}\times \partial D^2;\mathbb{Z})$ as a class of order $p$.  When $E=E(n)$ (and in a variety of other cases), this operation can be performed in the symplectic category. The key is that $E(n)$ carries a symplectic cusp neighborhood: by \cite{Sy} and \cite{FS97} we can carry out the operation by blowing up a point on a suitable singular fiber of $E(n)\to S^2$; then performing $p-2$ additional blowups; and finally symplectically rationally blowing down \cite{Sy} an appropriate configuration of $p-1$ symplectic spheres.  The result of this is another symplectic manifold $E(n)_p$.  Performing the operation once again for some other positive integer $q$ produces a manifold $E(n)_{p,q}$.

Now if we start with one of the above-produced aperiodic symplectic forms on $E(n)$, all of the above surgery operations can be performed in a region disjoint from the hypersurface  $Y$ which violates the nearby existence theorem; indeed, $Y$ is the preimage of a circle of regular values of the elliptic fibration $E(n)\to S^2$, and the logarithmic transformations are performed by surgeries that occur within neighborhoods of singular values of the fibration.  (Care should however be taken to keep the perturbation parameter $u$ in the aperiodic symplectic form $\omega_u$ small enough so that the relevant singular fibers will still be symplectic.)  Thus the hypersurface $Y$ survives as a hypersurface which violates the nearby existence theorem in any $E(n)_{p,q}$ with $n\geq 2$, and so all of these manifolds admit aperiodic symplectic forms.  
\end{ex}

\begin{ex}[Exotic elliptic surfaces] In \cite{FS}, Fintushel and Stern associate to any knot $K\subset S^3$ and any $n\geq 1$ a four-manifold $E(n)_K$ which is homeomorphic to $E(n)$, but which is necessarily non-diffeomorphic to $E(n)$ provided that the Alexander polynomial of $K$ is nontrivial.   In the special case that $K$ is a fibered knot (\emph{i.e.}, $S^3\setminus K$ admits a fibration over the circle), $E(n)_K$ naturally admits symplectic structures.  We observe here that, for $n\geq 2$ and $K$ fibered, $E(n)_K$ carries aperiodic symplectic forms.

Indeed, as noted just before \cite[Corollary 1.7]{FS}, when $K$ is fibered $E(1)_K$ can be constructed as follows.  Let $T$ and $T'$ be two fibers of the standard elliptic fibration $E(1)\to S^2$.  Let $M_K$ denote the result of $0$-framed surgery along $K\subset S^3$; then where $g$ is the Seifert genus of  $K$ the fact that $K$ is fibered implies that $3$-manifold $M_K$ admits a fibration $M_K\to S^1$ with a distinguished section corresponding to the core circle of the surgery and with fiber a surface of genus $g$.  Then we have a fibration $S^1\times M_K\to S^1\times S^1$ with a section $T_K$ which is a square-zero torus; the Thurston trick shows that $S^1\times M_K$ admits symplectic forms (obtained  by starting with a closed, fiberwise symplectic form and then adding a large multiple of the pullback of the standard symplectic form on the base $S^1\times S^1$) with respect to which the torus $T_K$ is symplectic.  Then $E(1)_K$ can be seen as the result of applying the symplectic sum construction with $M=E(1)\coprod (S^1\times M_K)$ and $i_+$ and $i_-$ being the inclusions of the tori $T'$ and $T_K$, where the symplectic form is scaled on the two components of $M$ to give $T'$ and $T_K$ equal area (we leave it to the reader to choose the bundle isomorphism $\Phi$; actually the choice will not affect the diffeomorphism type in this example).  

Recall that we specified another fiber $T$ of the fibration $E(1)\to S^2$; this fiber $T$ survives into $E(1)_K$ as a symplectic square zero torus.  In particular, for $n\geq 2$ we can form the symplectic sum of $E(n-1)$ and $E(1)_K$ along $T(n-1)$ and $T$. The result of this symplectic sum is equivalent to what Fintushel and Stern denote by $E(n)_K$, and as usual it contains a distinguished hypersurface $Y$.  Since $T(n-1)$ and $T$ are homologically linearly independent in $E(n-1)\coprod E(1)_K$ and since $E(n-1)\coprod E(1)_K$ is simply connected, Corollary \ref{4cor} and Lemma \ref{irrat} produce symplectic forms on $E(n)_K$ with respect to which $Y$ violates the nearby existence property; hence these forms are aperiodic.
\end{ex}

\begin{remark}\label{prodrem} If $n\geq 2$ and $K$ is any fibered knot (including the unknot), we have seen that $E(n)_K$ admits symplectic forms $\omega_{u,n,K}$ for which there is a hypersurface $Y(n,K)$ violating the nearby existence property.  It immediately follows from this that for any symplectic manifold $(P,\sigma)$, the hypersurface $Y(n,K)\times P\subset X\times P$ violates the nearby existence property with respect to the symplectic form $\omega_{u,n,K}\oplus \sigma$.  Thus we obtain many more examples of aperiodic symplectic forms in dimensions larger than $4$.  Specializing to $P$ equal to $S^2$, it follows as in \cite[Lemma 6.2]{IP99} that for any $K,K'$ there is a diffeomorphism $E(n,K)\times S^2\to E(n,K')\times S^2$ pulling back the cohomology class of $\omega_{u,n,K'}\oplus \sigma$ to that of $\omega_{u,n,K}\oplus\sigma$; moreover, when $K$ and $K'$ have the same Seifert genus this diffeomorphism intertwines the homotopy classes of almost complex structures induced by these symplectic forms.  However, as shown in \cite[Section 6]{IP99}, the symplectic deformation classes of these manifolds $E(n,K)\times S^2$ can be  distinguished from each other using Gromov--Witten invariants by the Alexander polynomial of the knot $K$.  There are infinitely many choices for this Alexander polynomial corresponding to any given knot genus that is larger than one.  Thus for $n\geq 2$ we have infinitely many symplectic structures on the smooth manifold $E(n)\times S^2$, each representing a distinct deformation class but representing the same cohomology class and inducing isotopic almost complex structures, all of which are aperiodic.
\end{remark}

\subsection{Arbitrary fundamental groups}\label{arbpi}
One of the most striking applications of the symplectic sum in Gompf's original papers \cite{Go94},\cite{Gom} on the symplectic sum was the fact that any finitely-presented group $G$ arises as the fundamental group of a symplectic four-manifold $X_G$.  We observe here that Gompf's manifold $X_G$ fits into our setup and so admits symplectic forms with respect to which the nearby existence property fails for some hypersurface.

Let us recall Gompf's construction, which appears in the proof of \cite[Theorem 4.1]{Gom} (and also in that of \cite[Theorem 3.1]{Go94}).  Starting with a finite presentation of $G$ (to avoid trivialities, if $G$ is trivial we assume that this is not the empty presentation), Gompf first constructs a symplectic form $\omega'$ on $F\times T^2$ for some surface $F$, and $\omega'$-symplectic tori $T_{i}\subset F\times T^2$ ($1\leq i\leq m$, where $m\geq 1$ since we assume the presentation to be nontrivial).  A bit more specifically, where $\alpha$ is a nontrivial circle on $T^2$, the $T_{i}\subset F\times T^2$ are embedded perturbations of the immersed tori $\gamma_i\times\alpha$ where the $\gamma_i$ are certain immersed loops in $F$ such that $G\cong \pi_1(F)/\langle \gamma_1,\ldots,\gamma_m\rangle$.  The torus $\{z\}\times T^2$ is also $\omega'$-symplectic, for suitable $z\in F$ chosen that this torus is disjoint from the $T_{i}$.  The symplectic manifold $X_G$ is formed by a sequence of symplectic sums with the rational elliptic surface $E(1)$: first form $X_{G}^{1}$ as the symplectic sum of $F\times T^2$ with $E(1)$ along  $\{z\}\times T^2\subset F\times T^2$ and the standard fiber $T(1)\subset E(1)$ (where the symplectic form on $E(1)$ has been rescaled to give $T(1)$ and $\{z\}\times T^2$ equal area); then, for $1\leq k\leq m$, form $X_{G}^{k+1}$ as the symplectic sum of $X_{G}^{k}$ with $E(1)$ (with an appropriately rescaled symplectic structure) along the tori $T_i\subset X_{G}^{k}$ and $T(1)\subset E(1)$.  Since $E(1)\setminus T(1)$ is simply connected, each summation affects the fundamental group by killing the inclusion-induced image of the fundamental group of the torus being summed along.  In particular we will have \[ \pi_1(X_{G}^{m+1})=\frac{\pi_1(F\times T^2)}{\langle \pi_1(\{z\}\times T^2),\gamma_i\times\alpha\rangle}=\pi_1(F)/\langle \gamma_1,\ldots,\gamma_m\rangle=G,\]
and so we may take the desired manifold $X_{G}$ equal to $X_{G}^{m+1}$.

To connect this to our construction, note that at each stage we are applying the symplectic sum to symplectic tori on different connected components of the manifold, so the tori are certainly homologically independent, allowing us to apply Corollary 
\ref{4cor}.  The inclusion-induced map $\pi_1(T(1))\to \pi_1(E(1))$ is of course trivial since $E(1)$ is simply connected.  Meanwhile, because $\pi_1(E(1)\setminus T(1))$ is trivial and because the first symplectic sum is performed along $\{z\}\times T^2\subset F\times T^2$, the inclusion-induced images $\pi_1(\{w\}\times T^2)\to \pi_1(M_{G}^{k})$ are trivial for all $k\geq 1$.  Recalling that $T_i\subset F\times T^2$ is obtained by perturbing an immersed torus of the form $\gamma_i\times \alpha$ where $\alpha$ is a homologically nontrivial curve in $T^2$, it follows that $\alpha$ gives an infinite-order element in the kernel inclusion-induced map $\pi_1(T_k)\to \pi_1(X_{G}^{k})$ for all $k\geq 1$. This proves that the group on the right-hand side of  (\ref{maindim4}) is nontrivial for the sums which form $X_{G}^{k+1}$ for all $k\geq 1$. Hence for all $k\geq 1$ the ``neck'' hypersurface $Y_{k+1}\subset X_{G}^{k+1}$ formed by the symplectic sum obeys the hypotheses of Lemma \ref{irrat}.  Applying this in particular when $k=m$ shows that Gompf's manifold $X_G$ admits aperiodic symplectic forms.

As is noted in the addendum to \cite[Theorem 4.1]{Gom}, $X_G$ contains a symplectic square zero torus $T$ for which the inclusion-induced map $\pi_1(T)\to \pi_1(X_G)$ is trivial; for instance one can take $T=\{w\}\times T^2$ for suitable $w$, and one can arrange this torus to be disjoint from the ``neck'' $Y_{m+1}$ of the last symplectic sum that was carried out to form $X_G$.  If the perturbation of the symplectic form $\omega_u$ that was used to make $Y_{m+1}$ violate the nearby existence property is small enough, then $T$ will still be $\omega_u$-symplectic.  If we then take the symplectic sum of $X_G$ with any other manifold $M$ along $T\subset X_G$ and any symplectic torus $T'\subset M$ such that $M\setminus T'$ is simply connected, then the resulting manifold $X_G(M)$ will still have fundamental group $G$ and will, as in Remark \ref{othersum},  have a hypersurface (a copy of $Y_{m+1}$) which violates the nearby existence property.  In \cite[Section 6]{Gom} this construction was used to construct symplectic manifolds with fundamental group $G$ having prescribed Euler characteristic and signature in certain ranges depending on $G$.  Thus these manifolds (including specifically the ones from \cite[Theorems 6.2, 6.3]{Gom}) all admit symplectic forms for which the hypersurface $Y_{m+1}$ violates the nearby existence property.  In particular, for any sufficiently large value of $e$ (with the bound depending on $G$), we get a number proportional to $e$ of homotopy-inequivalent symplectic manifolds with fundamental group $G$ and Euler characteristic $e$ in this fashion; as in \cite[Theorem 6.2]{Gom} these can be taken to be either spin or non-spin.\footnote{In fact, all of these examples are also minimal; this was not known for the non-spin examples when \cite{Gom} was written, but readily follows from \cite{U}.}

\subsection{Aperiodic geography} \label{geog} The symplectic geography problem asks which pairs of integers $(e,\sigma)$ can be realized as the Euler characteristic and signature of a minimal closed simply-connected symplectic four-manifold.  The symplectic sum has been a significant tool in addressing this problem, and often (though not always) the manifolds produced fit into our scheme and so admit aperiodic symplectic forms.  As examples, consider some of the manifolds from \cite{ABBKP}.  

Switching to the coordinates $c=3\sigma+2e$ and $\chi=\frac{e+\sigma}{4}$ that for historical reasons are generally used by geographers, \cite[Theorem 22]{ABBKP} produces for any pair $(\chi,c)$ of nonnegative integers with $0\leq c\leq 8\chi-2$ (with four exceptions) a minimal simply-connected symplectic four-manifold with characteristic numbers $c$ and $\chi$.  Nearly all of these admit aperiodic symplectic forms.  To be more specific, when either $c$ is even and $0\leq c\leq 8\chi-10$, or $c$ is odd and either $1\leq c\leq 8\chi-17$ or $7\leq c\leq 8\chi-11$, 
the manifold constructed in the proof of \cite[Theorem 22]{ABBKP} can be realized as a symplectic sum along square-zero tori $T_1,T_2$ of two manifolds $M_1$ and $M_2$, with $M_1$ simply connected and with the inclusion-induced map $\pi_1(T_2)\to \pi_1(M_2)$ having rank at most $1$.  Consequently Corollary \ref{4cor} and Lemma \ref{irrat} produce aperiodic symplectic forms on all of these manifolds.   

A similar remark applies to the manifolds produced in \cite[Section 9]{ABBKP}.  In particular, the proof of \cite[Theorem 24]{ABBKP} gives a manifold $S$ containing a symplectic square-zero torus $T$ having $\pi_1(S\setminus T)=0$, $\chi(S)=45$, and $c(S)=8\chi(S)+4=364$.  If we take the symplectic sum of $S$ with one of the manifolds $Z_1$ considered in the proof of \cite[Theorem 23]{ABBKP} along the torus $T\subset S$ and $T_2\subset Z_1$, the result will be simply connected and will admit aperiodic symplectic forms by Corollary \ref{4cor} and Lemma \ref{irrat}, and its characteristic numbers can be arranged to be $(c,\chi)=(364+c',45+\chi')$ for any values $c'$ and $\chi'$ with either $c'$ even and $0\leq c'\leq 8\chi'-2$ or $c'$ odd and $19\leq c'\leq 8\chi'-7$.\footnote{\cite{ABBKP} state that one does not need the restriction $c'\geq 19$; however some such restriction seems necessary, as for lower values of $c'$ some of the manifolds that they would use in the role of the manifold $Z_1$ (which are given in \cite[Section 7]{ABBKP}) do not contain suitable square-zero tori $T_2$.}  So we can realize any $(c,\chi)$ with $c$ even and $364\leq c\leq 8\chi+2$ or $c$ odd and $383\leq c\leq 8\chi-3$ in this way.  Instead summing $S$ with one of the manifolds $P_{1+2k,4+2k}$ at the end of \cite[Section 6]{ABBKP} and then perturbing the symplectic form produces aperiodic symplectic forms on manifolds with $385\leq c=8\chi+1$.  To get aperiodic symplectic manifolds with $c=8\chi-1$, use the manifold $B$ of \cite{ABBKP}, containing disjoint tori $T_1$ and $T_2$. First sum $B$ with $S$ along $T_2$ and $T$, which results in a simply-connected manifold $U$, and then sum $U$ with a suitable $P_{1+2k,4+2k}$; again we can perturb the symplectic form on this sum to an aperiodic one by  Corollary \ref{4cor} and Lemma \ref{irrat} since $U$ is simply connected and $\pi_1(P_{1+2k,4+2k})=\mathbb{Z}$. This sum can be arranged to have any characteristic numbers $(\chi,c)$ with $c$ odd and  $391\leq c= 8\chi-1$.

Combining all of the above, we have shown that there exist aperiodic symplectic forms on minimal four-manifolds representing all but finitely many points $(\chi,c)$ in the geography plane on or above the line $c=0$ (recall \cite[Theorem A]{Liu} that any minimal symplectic four-manifold which is not a ruled surface will have $c\geq 0$) and on or below the ``signature 2'' line $c=8\chi+2$.


\subsection{Examples where $Y$ is a nontrivial bundle}\label{nontrivy}  In all of the above four-dimensional examples, we took a symplectic sum along square-zero tori, yielding a hypersurface diffeomorphic to the three-torus in the sum which (for a perturbed symplectic form) violates the nearby existence property.  More generally, Corollary \ref{4cor} and Lemma \ref{irrat} allow us to sum along (equal-area) symplectic tori of opposite self-intersection $\pm k$; this yields a hypersurface $Y$ in the sum diffeomorphic to the a circle bundle over the torus with Euler number $k$.  After a perturbation of the symplectic form, $Y$ will violate the nearby existence property provided that the map $(i_+)_*-(i_-)_*\co H_1(T^2;\R)\to H_1(M;\R)$ vanishes.  

While our examples above all had $k=0$, many of them can be modified to have $k\neq 0$.  Namely, suppose that the input manifold $M$ is a disjoint union of the rational elliptic surface $E(1)$ and another symplectic four-manifold $U$, and that the tori being summed along are the standard fiber $T(1)\subset E(1)$ and a torus $T\subset U$ with trivial inclusion-induced map $H_1(T;\R)\to H_1(U;\R)$.  For example  any of the $E(n)_K$ or $E(n)_{p,q}$ with $n\geq 2$ can be obtained this way (take $U=E(n-1)_K$ or $E(n-1)_{p,q}$); so can many (though not all) of the manifolds from \cite{ABBKP}.  Also, if one starts with a sufficiently redundant presentation of an arbitrary finitely presented group $G$, the symplectic summation that produces the Gompf manifold $X_G$ with fundamental group $G$ will satisfy this description.   Recalling that $E(1)$ is obtained from $\mathbb{C}P^2$ by blowing up $9$ points on a cubic in $\mathbb{C}P^2$ (with $T(1)$ the proper transform of the cubic), we can blow down some number $k\leq 9$ of the exceptional divisors of these blowups to obtain a manifold $E_k\cong \mathbb{C}P^2\#(9-k)\overline{\mathbb{C}P^2}$ with a torus $F_k$ of self-intersection $k$.  At the same time, we can blow up $k$ points on $T\subset U$ to obtain a self-intersection-$(-k)$ torus $T_k\subset U_k$.  The area of $F_k$ will now be larger than that of $T_k$, but that can be repaired by rescaling the symplectic forms.  Moreover it will still be true that $H_1(F_k;\R)\to H_1(E_k;\R)$ and $H_1(T_k;\R)\to H_1(U_k;\R)$ both vanish.  Consequently the result $X_k$ of summing $E_k$ and $U_k$ along $F_k$ and $T_k$ will admit symplectic forms making the ``neck'' $Y_k$ violate the nearby existence property, where $Y_k$ is a principal $S^1$-bundle over $T^2$ with Euler number $k$.  The manifold $X_k$ is diffeomorphic to the manifold $X$ that would have been obtained without performing the $k$ blowups by \cite[Lemma 5.1]{Gom}; in fact the resulting symplectic forms on $X_k$ and $X$ are deformation equivalent by \cite[Proposition 1.6]{MSy}.  (It is not immediately clear whether we can arrange the perturbed forms obtained from Lemma \ref{irrat} to be equal.)

If one prefers examples with $k\neq 0$ that are not derived from examples with $k=0$ in the manner of the previous paragraph, one can easily construct them.  For instance, the elliptic surface $E(3)\to S^2$ has $9$ disjoint sections of self-intersection $-3$ each of which intersects the fiber once positively and transversely.  One can smooth these intersections to obtain a torus $T_3$ of self-intersection $-9$.  Then sum $E(3)$ and $\mathbb{C}P^2$ along $T_3$ and a cubic curve in $\mathbb{C}P^2$.

Finally, we mention an example of a hypersurface $Y$ in a symplectic manifold $(X,\omega)$ which is diffeomorphic to a principal $S^1$-bundle over the torus and violates the nearby existence property but provably \emph{cannot} be obtained by the construction of Lemma \ref{irrat}.  The manifold $X$ is diffeomorphic to the Kodaira--Thurston manifold of Example \ref{kt}.  Thus where $\phi\co T^2\to T^2$ is the map $(x,y)\mapsto (x,x+y)$, and $T_{\phi}=\frac{\mathbb{R}\times T^2}{(t+1,(x,y))\sim (t,\phi(x,y))}$ is the mapping torus of $\phi$, $X$ is diffeomorphic to $(\mathbb{R}/\mathbb{Z})\times T_{\phi}$.  Now the diffeomorphism type of $T_{\phi}$ is unaffected by an isotopy of $\phi$; accordingly choose an irrational number $a$ and define $\phi'\co T^2\to T^2$ by $\phi'(x,y)=(x+a,x+y)$.  Thus $\phi'\co T^2\to T^2$ is a symplectomorphism with respect to the standard symplectic structure $dx\wedge dy$ on $T^2$ which is isotopic to $\phi$ and has no periodic points (the torus is normalized so that $x$ and $y$ vary in $\mathbb{R}/\mathbb{Z}$).   There is a natural fiberwise symplectic form $\omega_{\phi'}$ on $T_{\phi'}$ characterized by the fact that it pulls back under the canonical projection $\mathbb{R}\times T^2\to T_{\phi'}$ to the form $dx\wedge dy$. Then where $X\cong (\mathbb{R}/\mathbb{Z})\times T_{\phi'}$ carries the symplectic form $ds\wedge dt+\omega_{\phi'}$, we let $Y=\{0\}\times T_{\phi'}$.  The characteristic foliation on any $\{s\}\times T_{\phi'}$ is generated by the vector field $\partial_t$; consequently all of its leaves pass through $\{t=0\}$ and the leaf through $(s,x,y)$ is closed iff $(x,y)$ is a periodic point of $\phi'$.  So since we chose $\phi'$ to have no periodic points $Y$ violates the nearby existence property.

Now, up to diffeomorphism, \[ Y=\frac{\mathbb{R}^3}{(t+1,x,y)\sim (t,x,x+y),(t,x+1,y)\sim (t,x,y+1)\sim(t,x,y)}.\]  The additive action of $\mathbb{R}/\mathbb{Z}$ on the third factor makes $Y$ into a principal $S^1$-bundle over $T^2$ with projection given by $(t,x,y)\mapsto (t,x)$.  (The Euler number of $p\co Y\to T^2$ is $-1$.)  Suppose that it were possible to obtain this hypersurface $Y$ violating the nearby existence property via the procedure of Lemma \ref{irrat}.  In particular, there would be some other symplectic form $\tilde{\omega}$ on $X$ such that $\tilde{\omega}|_Y=p^*\omega_{T^2}$ where $\omega_{T^2}$ is a symplectic form on $T^2$.  We show this is impossible:

\begin{prop}  There is no symplectic form $\tilde{\omega}$ on $X$ such that $\tilde{\omega}|_Y$ is the pullback by $p$ of a symplectic form on $T^2$.\end{prop}

\begin{proof}  Suppose that $\tilde{\omega}$ were such a symplectic form.  Performing the symplectic cut of $X$ along $Y$ would give another symplectic manifold $(M,\omega')$. 
Smoothly, $M$ is constructed from $X$ by cutting $X$ open along $Y$ to obtain a manifold $M^0$ with two boundary components diffeomorphic to $Y$ but with opposite orientations, and filling the two boundary components with the disc bundles $\pi_{\pm}\co N_{\pm}\to T^2$ having Euler numbers $\pm 1$; the zero sections $T_{\pm}$ of $N_{\pm}$ are symplectic tori in $M$ of square $\pm 1$, with \emph{equal area}, and the boundaries $\partial N_{\pm}$ have canonical identifications with the (unoriented) smooth manifold $Y$, with the restrictions of $\pi_{\pm}$ coinciding with the projection $p\co Y\to T^2$.  Now in our case $M^0$ is just $[0,1]\times Y$.  We can define a bundle projection $\pi\co M\to T^2$ by setting $\pi|_{N_{\pm}}=\pi_{\pm}$ and letting $\pi|_{[0,1]\times Y}$ equal the composition of the projection $[0,1]\times Y\to Y$ with $p\co Y\to T^2$.  If $(x,y)\in T^2$, then $\pi^{-1}(x,y)$ is the union of the discs $\pi_{\pm}^{-1}(x,y)$ and the cylinder $p^{-1}(x,y)\times [0,1]$; thus the fibers of $\pi$ are spheres and $M$ is the total space of an $S^2$-bundle over the torus.  

Now any $S^2$-bundle over the torus has $b_2=2$.  Since $T_{\pm}$ have self-intersection $\pm 1$, they represent linearly independent classes in $H_2(M;\R)$, so they span $H_2(M;\R)$.  The class $F=[T_+]-[T_-]$ has intersection number $+1$ with both $T_{\pm}$.  An easy linear algebra exercise shows that the facts that the symplectic form $\omega'$ has $\int_M\omega'\wedge\omega'>0$ and $\int_{T_+}\omega'>0 $ imply that $\int_{F}\omega'>0$, and hence that $T_+$ and $T_-$ have unequal area.  But for the manifold $M$ to have been obtained by the symplectic cut construction $T_{\pm}$ would need to have equal area, a contradiction.
\end{proof}

\section*{Appendix: Proof of Theorem \ref{llthm}}

We close the paper by assembling the facts from the literature that are necessary to prove the following theorem:

\begin{appthm}[Li-Liu] Any symplectic four-manifold with $b^+=1$ is $GW_g$-connected for some $g$.  Consequently all such symplectic four-manifolds have finite Hofer--Zehnder capacity.
\end{appthm}

\begin{proof} Let $(X,\omega)$ be such a manifold. Write $\kappa_X$ for the canonical class of $X$ (\emph{i.e.}, the negative of the first Chern class of $TX$ with respect to a compatible almost complex structure). Since $\kappa_X$ and the Gromov--Witten invariants are unchanged under deformation of the symplectic form we may without loss of generality assume that the de Rham cohomology class $[\omega]$ lies in the image of the coefficient-extension map $H^2(X;\Z)\to H^2(X;\R)$ and that (by rescaling $\omega$) \[ \langle [\omega]\cup [\omega],[X]\rangle \geq 4+b_1(X)+\langle [\omega]\cup \kappa_X,[M]\rangle.\]  Let $\gamma_1,\ldots,\gamma_{b_1(X)}$ be a basis for $H_1(X;\Z)/torsion$.  For any class $e\in H^2(X;\Z)\setminus \{0\}$ consider the modified Gromov--Taubes invariant \[ Gr'(e)(\gamma_1\wedge\cdots\wedge \gamma_{b_1(X)}) \] of \cite[Definition 3.4]{LL99}.  We also have corresponding Seiberg--Witten invariants $SW_{\pm}(e)(\gamma_1\wedge\cdots\wedge \gamma_{b_1(X)})$ associated to the spin$^c$ structure obtained as the tensor product of the canonical spin$^c$ structure $(X,\omega)$ with a line bundle of Chern class $e$, with the signs $\pm$ referring to the chamber determined by $\omega$.  As is summarized by \cite[Lemma 2.3, Theorem 2.9, and Lemma 3.3]{LL01}, Taubes' equivalence between $SW$ and $Gr$ \cite{Tbook} (as modified for the $b^+=1$ case by \cite{LL99}) along with the wall crossing formula and conjugation symmetry for Seiberg--Witten invariants, gives the following chain of equalities: \begin{align*}
Gr'(e)(\gamma_1\wedge\cdots\wedge \gamma_{b_1(X)})&=SW_-(e)(\gamma_1\wedge\cdots\wedge \gamma_{b_1(X)})
\\&=\pm SW_+(\kappa_X-e)(\gamma_1\wedge\cdots\wedge \gamma_{b_1(X)})\\&=\pm 1\pm SW_-(\kappa_X-e)(\gamma_1\wedge\cdots\wedge \gamma_{b_1(X)})\\&=\pm 1\pm Gr'(e)(\kappa_X-e)(\gamma_1\wedge\cdots\wedge \gamma_{b_1(X)}).\end{align*} 

Now specialize to $e=[\omega]$.  The invariant $Gr'(\kappa_X-[\omega])$ enumerates certain pseudoholomorphic representatives of the Poincar\'e dual to $\kappa_X-[\omega]$, possibly with several components some of which may be multiply covered; our modification of $\omega$ at the start of the proof implies that $PD(\kappa_X-[\omega])$ has negative symplectic area, so $Gr'(\kappa_X-[\omega])$ necessarily vanishes.  Consequently  $Gr'([\omega])(\gamma_1\wedge\cdots\wedge \gamma_{b_1(X)})=\pm 1$.  Now since $[\omega]$ evaluates positively on any symplectic surface (and in particular on any symplectic sphere of square $-1$), the modified Gromov--Taubes invariant $Gr'([\omega])$ coincides with the original Gromov--Taubes invariant $Gr([\omega])$ from \cite{Tbook}.  So we in fact have \[ 
Gr([\omega])(\gamma_1\wedge\cdots\wedge \gamma_{b_1(X)})=\pm 1 \] Where  $d([\omega])=\frac{1}{2}\langle [\omega]^2-\kappa_X\cup [\omega],[X]\rangle$, the latter invariant enumerates pseudoholomorphic representatives of $[\omega]$ which pass through generic $1$-cycles representing the $\gamma_i$ and through a generic set of $d([\omega])-\frac{1}{2}b_1(X)$ distinct points.  Note that our assumption on $[\omega]$ ensures that $d([\omega])-\frac{1}{2}b_1(X)\geq 2$.  

Recall that in general the pseudoholomorphic curves enumerated by the Gromov--Taubes invariant $Gr(e)$ can be disconnected, but are embedded except for the fact that some of their components may be multiply-covered square-zero tori (however such tori still miss all other components). However, as in \cite[Lemma 2.2]{M97}, for the particular case of $Gr([\omega])$ the curves in question \emph{will} be connected.  Indeed the curves cannot contain any exceptional sphere components since $[\omega]$ evaluates positively on all exceptional spheres.  Hence all components have nonnegative self-intersection, so their homology classes belong to the closure of the forward positive cone.  But then the light cone lemma shows that any two distinct components would need to intersect, contradicting embeddedness, unless all components were homologically proportional and had self-intersection zero, which is impossible since $\langle [\omega]\cup [\omega],[X]\rangle>0$.  Thus the curves in question can have only one component.

Now \cite{IP97} expresses $Gr([\omega])$ as a sum of contributions coming from various Gromov--Witten invariants, but the previous paragraph shows that all those contributions describing curves with more than one component necessarily vanish, leaving only the contribution of \[ GW_{g,d(\omega)+\frac{1}{2}b_1(X),PD[\omega]}([pt],\ldots,[pt],\gamma_1,\ldots,\gamma_{b_1(X)};[\bar{M}_{g,d(\omega)+\frac{1}{2}b_1(X)}]),\] where the genus $g$ is chosen via the adjunction formula to have the property that a genus-$g$ pseudoholomorphic representative of $PD[\omega]$ will be embedded: $g=1+\frac{1}{2}\langle [\omega]\cup[\omega]+\kappa_X\cup[\omega],[X]\rangle$.      Thus \[ GW_{g,d(\omega)+\frac{1}{2}b_1(X),PD[\omega]}([pt],\ldots,[pt],\gamma_1,\ldots,\gamma_{b_1(X)};[\bar{M}_{g,d(\omega)+\frac{1}{2}b_1(X)}])=\pm 1.\]  The number of appearances of $[pt]$ in this Gromov--Witten invariant is $d([\omega])-\frac{1}{2}b_1(X)$, which is at least two, so this completes the proof.
\end{proof}

\end{document}